\documentclass[12pt]{amsart}

\usepackage{amsthm}
\usepackage{amssymb}
\usepackage{amsmath}
\usepackage{amsfonts} %This is to have nice letters
\usepackage{amsthm}
\usepackage{bm}
\usepackage{color}
\usepackage{graphicx}
\usepackage[all]{xy}
\usepackage{mathtools}

\usepackage{euscript}
\usepackage{mathrsfs}

\newcommand{\excise}[1]{}%{$\star$\textsc{#1}$\star$}

\newtheorem{theorem}{Theorem}[section]
\newtheorem{lemma}[theorem]{Lemma}

\newtheorem{corollary}[theorem]{Corollary}
\newtheorem{proposition}[theorem]{Proposition}

\theoremstyle{definition}

\newtheorem{remark}[theorem]{Remark}

        {\begin{list}
                {\noindent\makebox[0mm][r]{\arabic{enumi}.}}
                {\leftmargin=5.5ex \usecounter{enumi}}
        }
        {\end{list}}

        {\begin{list}
                {\noindent\makebox[0mm][r]{(\roman{enumi})}}
                {\leftmargin=5.5ex \usecounter{enumi}}
        }
        {\end{list}}

\newcommand{\baseRing}[1]{\ensuremath{\mathbb{#1}}}
\newcommand{\Z}{\baseRing{Z}}
\newcommand{\C}{\baseRing{C}}
\newcommand{\N}{\baseRing{N}}

\newcommand{\aex}{{\rm a.e.}}

\def\<{\langle}
\def\>{\rangle}
\def\0{\mathbf{0}}

\def\CC{{\mathbb C}}

\newcommand{\NN}{{\mathbb N}}

\def\RR{{\mathbb R}}

\def\ZZ{{\mathbb Z}}

\def\cD{{\mathcal D}}

\def\cH{{\mathcal H}}

\def\cM{{\mathcal M}}

\def\cO{{\mathcal O}}

\def\cQ{{\mathcal Q}}

\def\del{\partial}
\def\nsupp{\text{\rm nsupp}}

\def\sp{\text{\rm sp}}
\def\sol{\text{\rm Sol}}

\numberwithin{equation}{section}
\newcommand\restr[2]{{% we make the whole thing an ordinary symbol
  \left.\kern-\nulldelimiterspace % automatically resize the bar with \right
  #1 % the function
  \vphantom{\big|} % pretend it's a little taller at normal size
  \right|_{#2} % this is the delimiter
  }}

\begin{document}

\title{Gevrey expansions of hypergeometric integrals I}\thanks{First author partially supported by MTM2010-19336 and FEDER, FQM333}

\author{Francisco-Jes\'us Castro-Jim\'enez}
\address{Departamento de \'Algebra, Universidad de Sevilla, Espa\~{n}a, Spain.} \email{castro@algebra.us.es}
\thanks{}

\author{Michel  Granger}
\address{Universit\'e d'Angers, D\'epartement de Math\'ematiques, LAREMA, CNRS
UMR n. 6093, 2 Bd. Lavoisier, 49045 Angers, France.}
\email{michel.granger@univ-angers.fr}
\thanks{}

\begin{abstract}
We study integral representations of the Gevrey series solutions of irregular hypergeometric
systems. In this paper we consider the case of the systems associated with a one row matrix,
for which the integration domains are one dimensional. We prove that any Gevrey series solution along the singular support of the system is the asymptotic expansion of a holomorphic solution given by a carefully chosen integral representation.
\end{abstract}

\maketitle

\begin{center}
%\currenttime \; \;
\today
\end{center}

\section{Introduction}

Hypergeometric systems also known as GKZ systems were introduced in \cite{GGZ87} and \cite{GZK89} as a far reaching generalisation of the Gauss hypergeometric differential equation. They appear as a special family of $D$-modules and they have been first studied in the regular case. For example in \cite{GKZ90} the authors consider integral representations of the solutions of hypergeometric systems, at generic points of the space, which they call Euler integrals. In the irregular or confluent case, A. Adolphson considers in  \cite{Adolphson} other integral representations of solutions which involve exponentials of polynomial functions and appropriate integration cycles. In this paper we develop new aspects in
the irregular case namely the link between Gevrey series solutions and holomorphic solutions in sectors following Adolphson's approach. We want to materialise a Gevrey series solution as an asymptotic expansion in a sector of such an integral solution.

%In \cite{GKZ90} the authors have studied integral representations of the solutions
%(also known as hypergeometric functions) of hypergeometric systems at generic points of the space.
%These systems have been introduced in \cite{GGZ87} and \cite{GZK89} developing the fact that
%they belong to a special case of $D$-modules. They have been first studied in the regular case
%and later in the irregular or confluent case. In this paper we develop new aspects in
%the irregular case namely the link between Gevrey series solutions and
%holomorphic solutions in sectors. These solutions are described by appropriate integral representations involving exponentials of polynomial functions and we aim at proving that we can materialise any Gevrey series as an asymptotic expansion in a sector of such a solution.

Let us fix some notations: $D$ stands for the complex Weyl algebra of order $n$, where
$n\geq 0$ is an integer. Elements in $D$ are linear partial
differential operators with polynomial coefficients. The polynomial ring $\C[\partial]:=\C[\partial_1,\ldots,\partial_n]$ is a subring of the Weyl algebra $D$, where the $\partial_j$'s  represent the partial derivatives with respect to the variables in the space $\CC^n$.
The input data for a GKZ system is a pair $(A,\beta)$ where $\beta$ is a vector in $\CC^d$
and $A=(a_{ij})=(a^{(1)},\cdots , a^{(n)})\in \Z^{d\times n}$ is a matrix of rank $d$ whose $j^{\rm{th}}$ column is $a^{(j)}$. The toric ideal $I_A \subset \C[\partial]$ is the ideal generated by the family of binomials $\partial^u -\partial^v$ where $u,v\in \N^n$ and $Au=Av$. The ideal $I_A$ is a prime ideal and the Krull dimension of the quotient ring $\CC[\partial]/I_A$ equals $d$.
Following \cite{GGZ87, GZK89}, the hypergeometric ideal associated with the pair $(A,\beta)$ is~: $$H_A(\beta)=DI_A + D(E_1-\beta_1,\ldots,
E_d-\beta_d)$$ where $E_i= \sum_{j=1}^n a_{ij}x_j\partial_j$
is the $i^{\rm{th}}$ Euler operator associated with the $i^{\rm{th}}$ row of $A$. The
corresponding hypergeometric $D$--module (or hypergeometric system) is the
quotient $D$--module $M_A(\beta):= \frac{D}{H_A(\beta)}$. %The variables $x_j$ are in 1-1 correspondence with the columns $a(j)$.

In \cite{GZK89} and \cite{Adolphson} it is proven that any hypergeometric $D$-module is holonomic. Moreover, a characterization
of the regularity of $M_A(\beta)$ is provided in the series of papers \cite{Hotta}, \cite{SST} and \cite{schulze-walther}:
The holonomic $D$-module  $M_A(\beta)$ is regular if and only if the toric ideal $I_A$ is homogeneous for the standard grading
in the polynomial ring $\CC[\partial]$. In particular the condition to be regular for $M_A(\beta)$ is independent of $\beta$.
The concept of regularity has been studied first in the case of an ordinary differential equation given by an operator $P\in D$ for $n=1$.
Regularity is characterised according to Fuchs theorem
by the nullity of the irregularity number, an invariant  combinatorially defined from the Newton polygon of $P$. In \cite{Malgrange-1974} B. Malgrange proved that the irregularity is the dimension of the space of solutions at the origin of $\CC$ with values in the space $\CC[[x]]/\CC\{x\}$ of formal power series modulo convergent ones. Later J.P. Ramis gave
a refined version of this result in \cite{Ramis-1984} calculating the space of solutions of a given Gevrey index again by using the Newton polygon of the operator.

The concept of irregularity in higher dimension is considerably more involved but generalizes the above results. Let us consider the structure sheaf $\cO_X$ of a complex manifold $X$, the sheaf of rings $\cD_X$ of linear differential operators with holomorphic  coefficients on $X$  and denote $\cO_{\widehat{X\vert Y}}$ the completion of $\cO_X$ along a smooth hypersurface $Y$. In \cite{Mebkhout-positivite} Z. Mebkhout introduces for a holonomic $\cD_X$--module $\cM$ its irregularity complex along $Y$, ${\rm Irr}_Y(\cM)={\RR\cH}om_{\cD_X}(\cM,\cO_{\widehat{X\vert Y}}/\cO_{X\vert Y})$, and in \cite{Laurent-Mebkhout} the Gevrey filtration of this complex is introduced and related to other invariants of the system, the algebraic slopes in the sense of Y. Laurent.

In the case of hypergeometric $D$-modules the irregularity sheaves along coordinate subspaces, and Gevrey series solutions are studied and described in \cite{F} (see also \cite{FC2,FC1}).  Beforehand A. Adolphson \cite{Adolphson} gave a formula for the dimension of the space of holomorphic solutions at a generic point of the space and for a generic value of the parameter $\beta$, and he also described integral representations of solutions of these confluent hypergeometric systems. In the non-confluent case an analogous dimension formula was previously given in \cite{GZK89}. In \cite{Esterov-Takeuchi-2012} A. Esterov and K. Takeuchi prove that these generic solution spaces are in fact completely described by integral representations along rapid decay cycles as introduced by M. Hien in \cite{Hien07} and \cite{Hien09}.

We want to explore the more hidden link between these integral representations and the Gevrey series solutions described in \cite{F} and \cite{FC1,FC2}. In this paper we treat the case of a matrix with one row $A=(a_1,\ldots, a_n)$ with  $0<a_1<\cdots <a_n$ a list of co-prime integers. This is the case where the integration cycles are paths, and as explained below the significant Gevrey series are along the hyperplane $x_n=0$. We prove that any Gevrey series solution of the system can be obtained as the asymptotic expansion of an integral representation along a well chosen path. A specificity of the one row case is that the rank of the space of Gevrey solutions is independent of $\beta$ and we can treat all the values of $\beta$. For a generic $\beta$ we only use the space of rapid decay cycles and for special values of $\beta$, namely for $\beta \in \ZZ\setminus\left(\NN a_1+\cdots +\NN a_n\right)$ we  must add an exceptional path, without the rapid decay property.  The method of the proof is to reduce the statement to the case of the matrix $A=(a,b)$, using the fact that the restriction is compatible with  being a Gevrey solution as well as with taking integrals on a fixed path. In parallel we know by an argument using a Gevrey version of Cauchy-Kovalevskaya theorem that the dimension of the Gevrey solution space is $a_{n-1}$ \cite{FC1}. In the case $A=(a,b)$ the Gevrey expansion is with respect to the second variable $x_2$. Its coefficients depend on the variable $x_1$ and are holomorphic in some sector depending on the chosen paths. The main issue is to choose carefully a number of paths of integration that yield a basis of the space of Gevrey series solutions and share a common sector of convergence for these coefficients. This choice appears to be possible if we restrict the range in the argument of the variable $x_2$ around a given direction.

Here is a summary of the contents of this paper. In Section 2 general facts are given about Gevrey series solutions following \cite{F}. In the case of a one row matrix we know that the
characteristic variety of the hypergeometric system is the union $T^*_XX\cup T^*_YX$ of the zero
section and of the conormal to the hypersurface $Y: (x_n=0)$. It is therefore sufficient to consider
the irregularity sheaf and the germs of Gevrey series solutions with respect to the hyperplane $Y$ and at a generic point
$(0,\dots,0,\epsilon,0)$ of $Y$. We recall   from \cite{FC1}  the description of a basis of
the space of Gevrey series solutions in terms of $\Gamma$-series and the proof that the dimension of this space
is always equal to $a_{n-1}$. We also describe its behaviour by the restriction operation which
consists in omitting variables among $x_1,\dots, x_{n-2}$. In Section 3 we recall the definition of
hypergeometric integrals of exponential type and the fact that %in the irregular case
they are
solutions of the system. It is a quite general fact that these solutions have an asymptotic
expansion as a Gevrey series with respect to the variable $x_n$ provided that we choose an
integration path of rapid decay both for the function $t^{-\beta-1}\exp\left(x_1t^{a_1}+\dots+x_nt^{a_n}\right)$
and for $t^{-\beta-1}\exp\left(x_1t^{a_1}+\dots+x_{n-1}t^{a_{n-1}}\right)$. At the beginning of Section 4 we prove that these
expansions are Gevrey series of order less than or equal to $\frac{a_n}{a_{n-1}}.$

In the remainder of this Section we prove the main result of this paper about realisation of these
Gevrey series as asymptotic expansion of integral solutions. First we treat in detail the case of dimension 2 and the last Subsection consists in using various restriction operations on the integral that are compared to the analogue described in Section 2 for the Gevrey series.
In a short last section we deduce from the previous results an explicit description of the germ  of the irregularity sheaf along $Y$ at a generic point. We get a family of integral with asymptotic expansions that yields as expected a basis of the space of classes of Gevrey series solutions modulo convergent ones.

% whose columns
%$a_1,\ldots,a_n$ span the $\Z$-module $\Z^d$. We also assume that
%all $a_i\neq 0$ and that the cone generated by the columns in $\R^n$
%contains no lines (one says in this case that this cone is {\em
%pointed}).

%An ideal $I$ in $\C[\partial]$ is said to be an $A$-graded ideal if
%it is generated by $A$-homogenous elements (equivalently if for
%every polynomial in $I$ all its $A$-graded components are also in
%$I$.)
%
%The matrix $A$ also induces a $\Z^d$-grading on the Weyl algebra $D$
%(also called the $A$-grading) by defining $\deg(\partial_i)=-a_i$
%and $\deg(x_i)=a_i$.

\section{Gevrey solutions of hypergeometric systems}\label{section-Gevrey-sol-hyp-sys} In this Section we review some results on the  construction of Gevrey series solutions of hypergeometric systems with respect to a coordinate hyperplane. We consider $X=\CC^n$ and the hyperplane $Y$ defined by $x_n=0$.
% $\cO_{X}$ the sheaf of holomorphic functions on $X$ and $\cD_X$ the sheaf of linear differential operators with holomorphic coefficients.
 With the hypergeometric system $M_A(\beta)$ we associate the left coherent $\cD_X$--module $\cM_A(\beta) := \frac{\cD_X}{\cD_X H_A(\beta)}$ which is called the analytic hypergeometric system associated with the pair $(A,\beta)$.

%Let us denote  and $\cO_{\widehat{X\vert Y}}$ the formal completion of $\cO_{X}$ along $Y$.
A germ $f$ of the sheaf $\cO_{\widehat{X\vert Y}}$ at a point $(p,0)\in Y$ has the form $f=\sum_{m\geq 0} f_m (x_1 ,\ldots , x_{n-1}) x_n^m$ where  all the $f_m$ are holomorphic functions in a common neighbourhood of $p$\,; in particular the restriction of $\cO_X$ to $Y$, denoted by $\cO_{X\vert Y}$, is a subsheaf of $\cO_{\widehat{X\vert Y}}$.

For any real number $s$, we consider the sheaf $\cO_{\widehat{X\vert Y}}(s)$ of {\em Gevrey series along $Y$} of order less than or equal to  $s$ defined as the subsheaf of $\cO_{\widehat{X\vert Y}}$ whose germs $f$ at any $(p,0)\in Y$
satisfy the following convergence condition: $$\sum_{m\geq 0} \frac{f_m (x_1 ,\ldots , x_{n-1})}{{m!}^{s-1}} x_n^m  \in \cO_{X \vert Y, (p,0)}.$$ %One has $\cO_{\widehat{X\vert Y}}(1)= \cO_{{X\vert Y}}$ and we write $\cO_{\widehat{X\vert Y}}(+\infty) = \cO_{\widehat{X\vert Y}}.$

If $s'<s$ then $\cO_{\widehat{X\vert Y}}(s')\subset\cO_{\widehat{X\vert Y}}(s)$.
%We denote by  $\cO_{\widehat{X\vert Y}}({<}\,\!s)$ the inductive limit sheaf of $\cO_{\widehat{X\vert Y}}(s')$ for $s'<s$.
If a germ $f$ belongs to $\cO_{\widehat{X\vert Y}}(s)$ for some $s$ but $f\notin \cO_{\widehat{X\vert Y}}(s')$ for all $s'<s$, we say that the {\em index} of the Gevrey series $f$ is $s$.

Let $A=(a^{(1)} \; \cdots \; a^{(n)} )$ be a full rank $d\times n$ matrix with $a^{(j)} \in
\ZZ^d$ for $j=1 ,\ldots ,n$.
In \cite{GZK89} and \cite{SST}, the authors associate with any vector $v\in \CC^n$ satisfying $Av=\beta$
a series expression of the form
\begin{equation}\label{def-Gamma-series-general}
\varphi_{A,\beta,v}(x):= x^{v} \sum_{u\in N_v} \Gamma[v;u]
x^{u} \end{equation}
where $N_v = \{u\in \ker_{\ZZ} (A)\; \vert  \;
\operatorname{nsupp}(v+u)=\operatorname{nsupp}(v) \}$,  $\ker_\ZZ
(A) =\{u\in \mathbb{Z}^n : \; A u=0 \}$ and
$\operatorname{nsupp}(w):=\{i\in \{1,\ldots , n\}\, \vert\,  w_i \in
\ZZ_{<0 }\}$ is the negative support of $w\in \CC^n$.
The coefficient $\Gamma[v;u]$ equals $\frac{[v]_{u_{-}}}{[v+u]_{u_{+}}}$ where
$[v]_{u}=\prod_{i} [v_i ]_{u_i} $ and  $[v_i ]_{u_i}=\prod_{j=1}^{u_i
} (v_i -j +1)$ is the Pochhammer symbol for $v_i \in \CC$, $u_i \in
\NN$. They call it the $\Gamma$-series associated with $v$.

We write $\varphi_v=\varphi_{A,\beta,v}$ is no confusion is possible.
It is proved in \cite[Proposition 3.4.13]{SST} that the formal expression $\varphi_v$ is annihilated by the hypergeometric ideal $H_A(\beta )$ if and only if the negative support of $v$ is minimal,
which means that  there is no $u \in \ker_\ZZ(A)$ with
$\operatorname{nsupp}(v+u)\subsetneq \operatorname{nsupp}(v)$.

When $\beta \in \CC^d$ is very generic, that is when $\beta $ is not in
a countable union of Zariski closed sets, there is a basis of the
Gevrey solutions space of $\cM_A (\beta )$ along $Y$ at a generic point of $Y$, given by series
$\varphi_v$ for suitable vectors $v\in \CC^n$, \cite[Th. 6.7]{F}.
%\ref{hypergeometric-Gevrey-solutions}).\marginpar{Maybe to be deleted}

In this article we restrict ourself to the case where $A=(a_1,\ldots,a_n)$ is
a row matrix with $0<a_1<\cdots <a_n$, $(a_1,\ldots ,a_n)$ coprime and $n\geq 2$.
We recall here some results about the Gevrey series solutions of $\cM_A(\beta)$ as presented in \cite[Sections 4, 5]{FC1}. Let us fix some notations. As before $X=\CC^n$ and $Y\subset X$ denotes the hyperplane defined by $x_n=0$. Let us write $Z\subset X$ the hyperplane defined by $x_{n-1}=0$.

When $s<\frac{a_n}{a_{n-1}}$  one sees from the results in \cite{FC1} that the set of Gevrey solutions in $\cO_{\widehat{X\vert Y}}(s)$ is zero if $\beta\notin\NN A$ and is a one dimensional space generated by a polynomial if $\beta\in \NN A$. This result follows also from what we show in this paper and we focus now on the case $s\geq\frac{a_n}{a_{n-1}}$.

\subsection{Case $a_1=1$}\label{case-a1-equal-1} Assume first $n\geq 3$. A basis of the free $\ZZ$--module $\ker_\ZZ(A)$ is formed of the vectors $\{u^{(2)},\ldots,u^{(n)}\}$ where   $u^{(n-1)} = (a_{n-1},0,\ldots,0,-1,0)$ and for $i=2,\ldots,n$, $i\not= n-1$, we define $$u^{(i)} = (-a_i,0,\ldots,0,{1},0,\ldots,0)$$ where $1$ is in the $i$-th component.
For each ${\bf m}=(m_2,\ldots,m_n)\in \ZZ^{n-1}$ we write $u({\bf m}) = \sum_{i=2}^n m_i u^{(i)}$ the corresponding element in $\ker_\ZZ(A)$.

For $j=0,\ldots,a_{n-1}-1$ define $v^j =(j,0,\ldots,0,\frac{\beta - j}{a_{n-1}},0)\in \CC^n$
and consider the associated $\Gamma$--series
\begin{equation}\label{def-Gamma-series-1-a2-an}
\varphi_{A,\beta,v^{j}} = x^{v^j} \sum_{\stackrel{m_2,\ldots, m_{n-1},m_n \geq 0}{ j+a_{n-1}m_{n-1} \geq {\sum_{i\neq n-1}
a_i m_i }}} \Gamma[v^j; u({\bf m})] x^{u({\bf m})}.\end{equation}

We  write $\varphi_{A,\beta}^{(j)} = \varphi_{A,\beta,v^{j}}$ and also $\varphi^{(j)} = \varphi_{A,\beta,v^{j}}$ if no confusion arises. By the choice of a basis of $\ker_\ZZ(A)$ that we make $\varphi^{(j)}$ is a series in $x^{v^j}\CC[[x_1,\ldots,x_{n-1}^{-1},x_n]][x_1^{-1}]$. %(\paco{$\pm 1$???}).

Notice here that the  summation in $\varphi^{(j)}$ is taken, according to (\ref{def-Gamma-series-general}), over the set $N_{v^{(j)}} = \{u({\bf m})\, \vert \, \nsupp(v^{(j)}+u({\bf m})) = \nsupp(v^{(j)}\}$. This set $N_{v^{(j)}}$ is indexed by
\[
\begin{cases}\{{\bf m} \in \NN^{n-1}\,\vert \, j+a_{n-1}m_{n-1} \geq \sum_{i\neq n-1} a_i m_i\} &\mbox{ if }\frac{\beta -j}
{a_{n-1}} \notin \NN
\\
\{{\bf m} \in \NN^{n-1}\,\vert \, j+a_{n-1}m_{n-1} \geq \sum_{i\neq n-1} a_i m_i,\; \frac{\beta -j}{a_{n-1}}\geq m_{n-1}\} &\text{ if }\frac{\beta -j}{a_{n-1}} \in \NN .\end{cases}
\]
In this last case we can write $\beta = j+ha_{n-1}$ for unique $0\leq j < a_{n-1}$ and $h\in \NN$ and then the series $\varphi^{(j)}$ is a polynomial since $[h]_{m_{n-1}}=0=\Gamma[v^j;u({\bf m})]$ if $m_{n-1}\geq h+1$. In fact $\varphi^{(j)}$ is a polynomial if and only if  we are in that case.

The negative support of each  $v^j$ is $\emptyset$ if $h\in \NN$ and $ \{n-1\}$ if $\beta - j$ is a negative multiple of $a_{n-1}$ and it is minimal in both cases. The series $\varphi^{(j)}$ is a polynomial in the first case and a solution of the ideal $H_A(\beta)$ in both. More precisely,

\begin{theorem}\label{basis-gevrey-solutions} {\rm \cite[Th. 4.21, i)]{FC1}} Let $A=(1,a_2, \ldots, a_n )\in \ZZ^{n}$ with $1<a_2 <\cdots < a_n $, $Y=(x_n=0)\subset X$
and $Z=(x_{n-1}=0)\subset X$. Then
the set of germs at $p$ of Gevrey series $\{\varphi^{(j)} \, \vert \, j=0,\ldots,a_{n-1}-1\}$ is a basis of
$\cH om_{\cD_X}(\cM_A(\beta),\cO_{\widehat{X|Y}}(s))_p$ for all
$\beta \in \CC$, $p\in Y\setminus Z$ and $s\geq a_n /a_{n-1}.$
\end{theorem}

It is also useful to consider the immersion $i: \CC^3 \hookrightarrow X$ defined by the equations $x_2=\cdots=x_{n-2}=0$ and the restriction $\rho'$ with respect to this immersion (coordinates in $\CC^3$ are $(x_1,x_{n-1},x_n)$).

The series $\varphi^{(j)}_{A,\beta}(x_1,0,\ldots,0,x_{n-1},x_n) $ equals precisely $\varphi^{(j)}_{{(1,a_{n-1},a_n),\beta}}(x_1,x_{n-1},x_n)$ so the restriction defines an isomorphism
\begin{equation}\label{restriction-x2-x_{n-2}} \cH om_{\cD_X}(\cM_A(\beta),\cO_{\widehat{X|Y}}(s)) \xrightarrow[\phantom{xxxxxx}]{\rho'}
\cH om_{\cD_{\CC^3}}(\cM_{(1,a_{n-1},a_n)}(\beta),\cO_{\widehat{\CC^3|Y_1}}(s)) \end{equation} of the corresponding stalks at any point in  $Y_1\setminus Z_1$ for all $\beta \in \CC$ and $s\in \RR$ where we denote $Y_1,Z_1$ the subspaces in $\CC^3$ with equations $x_n=0$ and $x_{n-1}=0$.

\begin{remark} So far we have assumed $n\geq 3$. The case $n=2$ is special and will be treated now, following \cite{FC2}. We can drop in this case the assumption on $a_1$ and simply write $a=a_1, b=a_2$, with $1\leq a<b$ and $\gcd(a,b)=1.$

The corresponding $\Gamma$--series have a slightly different shape (see \cite{FC2}): For $j=0,\ldots,a-1$, consider $w^{j}=(\frac{\beta - j b}{a},j)\in \CC^2$ and
\begin{equation} \label{def-Gamma-series-a-b} \psi_{A,\beta}^{(j)} = x^{w^{j}} \sum_{m\geq 0} \frac{[\frac{\beta -j b}{a}]_{bm}}{[am+j]_{am}} x_1^{-bm} x_2^{am}.\end{equation} \end{remark}

We have the following
\begin{proposition}\label{basis-gevrey-solutions-dim-2} {\rm \cite[Prop. 5.3 and 5.4]{FC2}}
Write $X=\CC^2$, $Y=(x_2=0)\subset X$. The set (of germs) of Gevrey series  $\{\psi_{{A},\beta}^{(j)}\, \vert \, j=0,\ldots, a-1\}$ is a basis of the stalk of $\cH om_{\cD_{X}}(\cM_{A}(\beta), \cO_{\widehat{{X} \vert { Y}}}(s))$ at any point in ${Y}\setminus \{(0,0)\}$, for any real number $s\geq \frac{b}{a}$ and any $\beta \in \CC$.
\end{proposition}

For a generic value of the parameter $\beta$, $\varphi_{(1,a,b)}^{(j)}(x_0,x_1,x_2)$ restrict by setting $x_0=0$ to another basis of the solution space in Proposition \ref{basis-gevrey-solutions-dim-2}. The shape of the base change is a reindexation $j\to j'$ composed with a diagonal invertible matrix $\varphi_{(1,a,b)}^{(j)}(0,x_1,x_2)=\lambda_j \psi_{A,\beta}^{(j')}$. This fact and more, is explained in Proposition \ref{varpi-for-k-geq-1} 1) below.

\subsection{Case $a_1>1$}  Recall that $X=\CC^n$, $Y \subset X$ (resp. $Z \subset X$) is defined by $x_n=0$ (resp. $x_{n-1}=0$). First of all, we follow \cite[Rk. 5.4]{FC1} to prove the following equality
\begin{proposition}
\begin{equation}\label{dim=a_n-1}
\dim_\CC\left(\cH om_{\cD_{X}}(\cM_{A}(\beta), \cO_{\widehat{X\vert Y}}(s))_p\right) = a_{n-1} \end{equation} if $p\in Y \setminus Z$ and $s\geq \frac{a_n}{a_{n-1}}$ for any $\beta \in \CC$.
\end{proposition}

\begin{proof}
We apply, among other results, Cauchy-Kovalevskaya's Theorem for Gevrey series. We consider $A'=(1,\, a_1, \ldots, a_n)$ and the hypergeometric system $\cM_{A'}(\beta)$ on $X':=\CC^{n+1}$ with coordinates $(x_0,x_1,\ldots,x_n)$.
We denote by $Y'\subset X'$ (resp. $Z'\subset X'$) the hyperplane $x_n=0$ (resp. $x_{n-1}=0$) and we identify  $X\subset X'$ with the hyperplane $x_0=0$. Notice that $Y=Y'\cap X$ and $Z=Z'\cap X$.

By  \cite[Proposition 4.2]{Castro-Takayama} we can apply Cauchy-Kovalevskaya's Theorem for Gevrey series solutions (see \cite[Corollary 2.2.4]{Laurent-Mebkhout2}) to deduce that there exists a Cauchy-Kovalevskaya's isomorphism $CK^s_{X',X}$
$$ \cH om_{\cD_{X'}}(\cM_{A'}(\beta), \cO_{\widehat{X'\vert Y'}}(s))_{| X} \xrightarrow[\phantom{xxxxxxxxx}]{CK^s_{X',X}} \cH om_{\cD_{X}}(\cM_{A'}(\beta)_{| X}, \cO_{\widehat{X\vert Y}}(s))$$ where $\cM_{A'}(\beta)_{| X}$ stands for the restriction in the category of $\cD$--modules. This is true for any $s\in \RR$ and for any $\beta$ and even if  $\gcd(a_1,\ldots,a_n)\neq 1$.
%\paco{For applying CK Theorem We do not need $\gcd{a_1,\ldots,a_n}= 1$.}

We write $CK^s=CK_{X',X}^s$ is no confusion is possible. The isomorphism $CK^s$ is induced by the inclusion $X \subset X'$,
  by the action of restriction on the modules involved.
% and so if $\varphi(x_0,x_1,\ldots,x_n)$ is a solution in the first space at a point $p\in X$, we have $CK^s(\varphi)(x_1,\ldots,x_n) = \varphi(0,x_1,\ldots,x_n)$. \paco{double check!}.
Theorem \cite[Th. 5.1]{FC1}
%\paco{This theorem uses $\gcd(a_1,\ldots,a_n)=1$}
states that for any $\beta\in \CC$ there exists $\beta'\in \CC$ such that the restriction $\cM_{A'}(\beta)_{| X}$ is isomorphic to the hypergeometric $\cD_X$--module $\cM_{A}(\beta')$. Moreover, by the same Theorem \cite[Th. 5.1]{FC1} one can take $\beta'=\beta$ for all but finitely many $\beta$. We denote by $\sp(A)$ the finite set of $\beta\in \CC$ such that $\beta' \neq \beta$. In particular, the isomorphism $CK^s$ and  Theorem \ref{basis-gevrey-solutions} prove equality (\ref{dim=a_n-1}) for $p\in Y\setminus Z$, $s\geq \frac{a_n}{a_{n-1}}$ and $\beta \not\in \sp(A)$.

Assume now $\beta^*\in \sp(A)$. We can take $\beta = \beta^* +A'\gamma'=\beta^* + A\gamma$ for a suitable $\gamma'=(0,\gamma)\in \NN^{n+1}\cup (-\NN)^{n+1}$ in such a way that $\beta \not\in \sp(A)$ and so the corresponding morphism  $CK^s$ is an isomorphism. By using \cite[Th. 2.1]{Saito01}, \cite[Th. 6.5]{Berkesch} one has that the morphism  $\cdot\, \partial^\gamma : \cM_A(\beta^*) \rightarrow \cM_A(\beta)$ if $\gamma \in \NN^n$ (resp. $\cdot\,\partial^{-\gamma} : \cM_A(\beta) \rightarrow \cM_A(\beta^*)$ if $\gamma \in (-\NN)^n$) is an isomorphism. This proves equality (\ref{dim=a_n-1}) for $\beta^* \in \sp(A)$.
\end{proof}

Once the equality (\ref{dim=a_n-1}) is established we give a description of a basis of the solution space of $\cM_{A}(\beta)$ in $\cO_{\widehat{X\vert Y}}(s))$. We use the restriction morphism
\begin{equation}\label{restriction-1-a1-an-to-a1-an}
\cH om_{\cD_{X'}}(\cM_{A'}(\beta), \cO_{\widehat{X'\vert Y'}}(s))_{|X} \xrightarrow[\phantom{xxxxxxx}]{\rho^s_{X',X}}\cH om_{\cD_{X}}(\cM_{A}(\beta), \cO_{\widehat{X\vert Y}}(s))\end{equation}
for $\beta \in \CC$ and  $s\in \RR$. This morphism is well defined and it is induced by the restriction to $x_0=0$. It is useful to write $x'=(x_0,x)$ and $x=(x_1,\ldots,x_n)$. Then $\rho^s_{X',X}(\varphi(x'))=\varphi(0,x)$, since if $\varphi(x')$ is a solution in the first space then $\varphi(0,x)$ is a solution in the second one. Notice that it is an approach of restriction that is different from the one by the $CK$'s.

We now consider the basis $\{\varphi^{(j)}_{A',\beta}(x_0,x)\}_{j=0}^{a_{n-1}-1}$ of germs (at a point in $Y\setminus Z \subset Y'\setminus Z'$) of Gevrey series solutions (of order $\leq s$) of $\cM_{A'}(\beta)$ described in Theorem \ref{basis-gevrey-solutions}.% for  $s\geq \frac{a_n}{a_{n-1}}$ and for all $\beta \in \CC$.
We simply write $\varphi^{(j)}=\varphi^{(j)}_{A',\beta}$. The terms in ${x_{n-1}^{\frac{-\beta+j}{a_{n-1}}}} \varphi^{(j)} $ have the form $$\Gamma[v^{(j)};u({\bf m})]x_0^j(x')^{u({\bf m})}=\Gamma[v^{(j)};u({\bf m})] x^{j+a_{n-1}m_{n-1} - \sum_{i\neq n-1} a_i m_i}_0 x_1^{m_1}\cdots x_{n-2}^{m_{n-2}} x_{n-1}^{-m_{n-1}} x_n^{m_n}$$ where $u({\bf m})$ is a general element of $\ker_\ZZ(A')\subset \ZZ^{n+1}$. The  summation in $\varphi^{(j)}$ (see (\ref{def-Gamma-series-1-a2-an})) is taken over the set $\{{\bf m} \in \NN^{n}\,\vert \, j+a_{n-1}m_{n-1} \geq \sum_{i\neq n-1} a_i m_i\}$.

It is useful to write the formal expansion of $\varphi^{(j)}(0,x)$. According to what is said before we have

\begin{equation} \label{varphi(0,x)} \varphi^{(j)}(0,x) = x_{n-1}^{\frac{\beta -j}{a_{n-1}}} \sum_{\stackrel{m_1,\ldots, m_{n-1},m_n \geq 0}{j+a_{n-1}m_{n-1} = {\sum_{i\neq n-1}
a_i m_i}}} \frac{[\frac{\beta-j}{a_{n-1}}]_{m_{n-1}} j !}{m_1 !\cdots m_{n-2} !m_n!} x_1^{m_1}\cdots x_{n-2}^{m_{n-2}} x_{n-1}^{-m_{n-1}}x_{n}^{m_{n}}.\end{equation}

%if $\frac{\beta -j}{a_{n-1}} \not\in \NN$ and $N_{v^{(j)}} = \{{\bf m} \in \NN^{n}\,\vert \, j+a_{n-1}m_{n-1} \geq \sum_{i\neq n-1} a_i m_i, \frac{\beta -j}{a_{n-1}}\geq m_{n-1}\}$ if $\frac{\beta -j}{a_{n-1}} \in \NN$.

%In particular, if $\frac{\beta - j}{a_{n-1}}\not\in \NN$ then the coefficient $\Gamma[v^{(j)}; u({\bf m})]$ is not zero and the expression of $\psi^{(j)}$ (and a fortiori of $\varphi^{(j)}$) contains countably infinite many terms.

%The terms of the restriction ${x_{n-1}^{\frac{-\beta+j}{a_{n-1}}}} \varphi^{(j)}(0,x) $ have the form $\Gamma[v^{(j)};u({\bf m})] x_1^{m_1}\cdots x_{n-2}^{m_2} x_{n-1}^{-m_{n-1}} x_n^{m_n}$ for ${\bf m}$ in the set $\{{\bf m} \in \NN^{n}\,\vert \, j+a_{n-1}m_{n-1} = \sum_{i\neq n-1} a_i m_i\}$.
Notice that the family of non zero $\varphi^{(j)}(0,x)$ is $\CC$--linearly independent because their supports are pairwise disjoint.
%\paco{Maybe we should say now that there exists $j$ such that the corresponding restriction is zero if and only if $\beta \in \NN \setminus \NN A$}
This, and equality (\ref{dim=a_n-1}), proves the following (see \cite[Remark 5.4]{FC1})

\begin{theorem}\label{basis-phi(0,x)-generic-case}
Assume that $\varphi^{(j)}(0,x)$ is non zero for $j=0,\ldots,a_{n-1}-1$. Then the set (of germs of) Gevrey series $\{\varphi^{(j)}(0,x)\, \vert\, j=0, \ldots,a_{n-1}-1\}$ is a basis of the stalk of $\cH om_{\cD_{X}}(\cM_{A}(\beta), \cO_{\widehat{X\vert Y}}(s))$ at any point in $Y\setminus Z$ for  $s\geq \frac{a_n}{a_{n-1}}$.
\end{theorem}

\begin{remark}\label{basis-phi(0,x)-special-case}
Assume now that for some $j=0,\ldots,a_{n-1}-1$ one has $\varphi^{(j)}(0,x)=0$. By \cite[Rk. 5.4]{FC1} this condition happens if and only if $\beta \in \NN \setminus \NN A$.
% \paco{we add a proof in this preliminary version; the proof should be deleted in the last version}. Assume that there exists $j=0,\ldots,a_{n-1}-1$ such that $\varphi^{(j)}(0,x)=0$. Then, according to the expression (\ref{varphi(0,x)}), for any ${\bf m}\in \NN^n$ such that $j+a_{n-1}m_{n-1} = \sum_{i\neq n-1}
%a_i m_i $ one has $[\frac{\beta-j}{a_{n-1}}]_{m_{n-1}}=0$. In particular $h:=\frac{\beta-j}{a_{n-1}}\in \NN$ and $0 \leq h < m_{n-1}$. If $\beta \in \NN A$ then $\beta = \sum_{i=1}^{n} b_i a_i$ for some $b_i\in \NN$ and we can write $j+(h-b_{n-1})a_{n-1} = \sum_{i\neq n-1} b_i a_i$. One has $h-b_{n-1}\geq 0$ (otherwise the left hand side of the last equality is a strictly negative number while  the right hand side is not) and this contradicts our assumption because $[h]_{h-b_{n-1}}\not=0$. Reciprocally, if $\beta \in \NN \setminus \NN A$ the $\beta = j + a_{n-1}h$ for unique $0\leq j <a_{n-1}$ and $h \in \NN$.  Then $\varphi^{(j)}$ is a polynomial and moreover, if ${\bf m}\in \NN^n$ is such that $j+a_{n-1}m_{n-1}=\sum_{i\not= n-1} a_i m_i$, then $\beta = j+a_{n-1}h = a_{n-1}(h-m_{n-1}) + \sum_{i\not=n-1} a_i m_i$. So, since $\beta \not\in \NN A$, $h-m_{n-1} <0$ and then $[\frac{\beta-j}{a_{n-1}}]_{m_{n-1}}=0$. This implies that $\varphi^{(j)}(0,x_1,x_2)=0$.
Furthermore this $j$ is then unique for a fixed $\beta$ and the image of the morphism $\rho_{X',X}^s$ has codimension 1 in the stalk of $\cH om_{\cD_{X}}(\cM_{A}(\beta), \cO_{\widehat{X\vert Y}}(s))$ at any point in $Y\setminus Z$ and for all $s\geq \frac{a_n}{a_{n-1}}$. We will show in Theorem \ref{a.e.are-basis-a1-an} how to describe a basis of this last solution space for all $\beta\in \CC$.
\end{remark}

%\pacoverde{say something more about a basis of this last space...}
%To find a basis of this last vector space we proceed as before: there is $\gamma'=(0,\gamma)\in \NN^{n+1}$ such that $\beta'=\beta-A'\gamma' = \beta - A \gamma\in \ZZ_{<0}$. We consider the following diagram:
%
%
%$$\xymatrixcolsep{4pc}
%\xymatrix{
%\cH om_{\cD_{X'}}(\cM_{A'}(\beta'), \cO_{\widehat{X'\vert Y'}}(s))_p \ar[r]^{{\rho^s_{X',X}}} &
%\cH om_{\cD_{X}}(\cM_{A}(\beta'), \cO_{\widehat{X\vert Y}}(s))_p  \\ &
%\cH om_{\cD_{X}}(\cM_{A}(\beta), \cO_{\widehat{X\vert Y}}(s))_p \ar[u]^{\overline{\cdot \partial^\gamma}}}$$
%
%The horizontal morphism is an isomorphism by Theorem \ref{basis-phi(0,x)-generic-case} and because $\beta'\not\in \NN\setminus \NN A$. The vertical morphism is the isomorphism associated with the isomorphism $\cdot \partial^\gamma : M_A(\beta') \longrightarrow M_A(\beta)$. This last isomorphism follows from the fact that $\beta, \beta' \in \ZZ \setminus \NN A$ and by applying
%\cite[Th. 2.1]{Saito01}, \cite[Th. 6.5]{Berkesch}. So the family $$\{(\overline{\cdot \partial^\gamma})^{-1}(\varphi^{(j)}_{A',\beta'}(0,x))\, \vert \, j=0,\ldots, a_{n-1}-1\}$$ is a basis of $\cH om_{\cD_{X}}(\cM_{A}(\beta), \cO_{\widehat{X\vert Y}}(s))_p$ for $p\in Y\setminus Z$ and $s\geq \frac{a_n}{a_{n-1}}$.
%\marginpar{When $s< \frac{a_{n}}{a_{n-1}}$ ...}

\subsection{The restriction to $x_0=0$ for $A'=(1,\, ka, \, kb)$}\label{restriction-to-x0}
%\paco{Take care of $A$, $A'$ ... notations. This part could be deleted later}

Let us consider the morphism $\rho^s_{X',X}$ when $X=(x_0=0) \subset X'=\CC^3$ and the matrix $A'=(1,\, ka, \, kb)$ with $1\leq a < b$, $1 < ka$ and $\gcd(a,b)=1$.  The morphism $\rho^s_{X',X}$, as defined in (\ref{restriction-1-a1-an-to-a1-an}), sends solutions of $\cM_{A'}(\beta)$ to solutions of $\cM_{(ka, kb)}(\beta)$ and this last $\cD_X$--module is isomorphic to $\cM_{(a, b)} (\frac{\beta}{k})$. In particular if $k>1$, $\rho^s_{X',X}$ is not an isomorphism, since the corresponding solutions spaces have dimensions $ka$ and $a$ respectively.  Below we describe in detail a different but related morphism involving the restrictions of the derivatives up to order $k-1$ with respect to the variable $x_0$.

So, instead of considering $A=(ka, kb)$ as before it is better to write $A=(a,b)$. Coordinates in $X'$ are $(x_0,x_1,x_2)$. %\paco{Notice that we are using here a slightly different notation than the previous one concerning $A$ and $A'=(1,A)$ ...}

For $j=0,\ldots,ka-1$ let us write $v^{(j)} =(j, \frac{\beta -j}{ka},0)$ and $$\varphi_{A',\beta}^{(j)}= x_1^{\frac{\beta -j}{ka}} \sum _{\stackrel{m_1,m_2 \geq 0}{j+kam_1\geq kbm_2}} \frac{[\frac{\beta-j}{ka}]_{m_1}j!}{m_2! (j+kam_1-kbm_2)!} x_0^{j+kam_1-kbm_2} x_1^{-m_1}x_2^{m_2}.$$

Notice that if a series $\varphi$ (as for example $\varphi_{A',\beta}^{(j)}$) is a solution of $\cM_{A'}(\beta)$ then $\restr{\frac{\partial^\ell \varphi}{\partial x_0^\ell}}{x_0=0}$ is a solution of the system $\cM_{A}(\frac{\beta-\ell}{k})$ %for $j=0,\ldots, ka-1$ and
for all $\ell\geq 0$. %$r=0,\ldots,k-1$.
We are going to compare this solution, for $\varphi=\varphi_{A',\beta}^{(j)}$, to the usual Gamma series solutions $\psi^{(j')}_{A,\beta'}(x_1,x_2)$, see (\ref{def-Gamma-series-a-b}),  of this last system $\cM_A(\beta')$ for  $\beta' = \frac{\beta -\ell}{k}$.
We consider the following $\CC$--linear map

\begin{equation} \label{varpi} \cH om_{\cD_{X'}}(\cM_{A'}(\beta), {\cO_{\widehat{X'\vert Y'}}}(s))_{\vert X}  \xrightarrow[\phantom{xxxxxxx}]{\varpi_\beta^s} \bigoplus_{\ell=0}^{k-1} \cH om_{\cD_{X}}\left(\cM_{A}\left(\frac{\beta-\ell}{k}\right), {\cO_{\widehat{X\vert Y}}}(s)\right)\end{equation}
which maps $\varphi$ to the vector $\left(\frac{\partial^{\ell}\varphi}{\partial x_0^{\ell}}(0,x_1,x_2)\right)_{\ell=0}^{k-1}$.
Here $Y'=(x_2=0)\subset X'$ and $Y=X\cap Y'$.

Notice that the morphism $\varpi_\beta^s$ is well defined and that it coincides with
the morphism (\ref{restriction-1-a1-an-to-a1-an}) when $k=1$. The following Proposition generalizes to an arbitrary $k$ what is already proved for $k=1$ in Theorem \ref{basis-phi(0,x)-generic-case} and in Remark \ref{basis-phi(0,x)-special-case}~:

\begin{proposition}\label{varpi-for-k-geq-1} 1) Assume that $\beta \not\in \NN$ or $\beta \in \bigcup_{r=0}^{k-1} (r + k \NN A)$. Then the $\CC$--linear map $\varpi_\beta^s$  is an isomorphism of vector spaces.

2) Assume $\beta \in \bigcup_{r=0}^{k-1} (r + k (\NN \setminus \NN A))$ and consider the integers $j_0,r_0,q_0$ uniquely determined by $0\leq j_0<ka$, $0\leq r_0 < k$  and  $\beta=j_0+kah=r_0+kq_0+kah$. Then the image of the map $\varpi_\beta^s$ is the codimension $1$ subspace, generated by all the $(0,\dots, 0, \psi^{(j')}_{A,\frac{\beta-r}{k}},0,\ldots , 0)$, with a non zero term in position $r$, and $(j',r)\neq (j'_0,r_0)$ for $j'_0\in \{0,\dots,a-1\}$ uniquely determined by $bj'_0\equiv q_0 \; \pmod a$.
\end{proposition}

%Let us write for $p=0,\ldots,a-1$, $w^{p} =(\frac{\beta' -pb}{a},p)$ and $$\varphi_{A,\beta'}^{(p)}= x_1^{\frac{\beta' -pb}{a}} \sum _{\stackrel{m \geq 0}{}} \frac{[\frac{\beta'-pb}{a}]_{bm}}{[am+p]_{am}} x_1^{-bm} x_2^{am}.$$

\begin{proof} We consider the morphism $\varpi_\beta^s$ induced on the stalks at points in $Y\setminus Z$, $Z=(x_1=0)\subset X$ and $s\geq \frac{b}{a}$.
Let us fix $j=0,\ldots,ka-1$ and write $j=kq+r$ for unique $0\leq q <a$ and $0\leq r < k$. So we have $$\frac{\beta -j}{ka} = \frac{\beta -kq-r}{ka} = \frac{\beta -r}{ka} - \frac{q}{a} = \frac{\frac{\beta -r}{k}-q}{a}.$$

The general exponent of $x_0$ in $\varphi_{A',\beta}^{(j)}$ is $j+k(am_1-bm_2) = r + k(q+am_1-b{m_2})$.
%For each $m_2\geq 0$ we write $m_2=am+p$ for $0\leq p \leq a-1$.
An exponent of $x_1$ in $\restr{\frac{\partial^\ell \varphi_{A',\beta}^{(j)}}{\partial x_0^\ell}}{x_0=0}$ for $0\leq \ell < k$, can only come from a term  in $\varphi_{A',\beta}^{(j)}$ for which the underlying exponent of $x_0$ is $\ell$. Therefore the only pairs $(m_1,m_2)\in \NN^2$ that may appear are those satisfying the relation~:
\begin{equation}\label{exponent}
j+k(am_1-bm_2)-\ell=k(q+am_1-bm_2)+r-\ell=0.
\end{equation}
This cannot happen if $\ell\neq r$ which means that $\restr{\frac{\partial^{\ell} \varphi_{A',\beta}^{(j)}}{\partial x_0^{\ell}}}{x_0=0}=0$ if $\ell \neq r.$ The equality (\ref{exponent}) is equivalent to $q=bm_2-a m_1$ and  $\ell=r$. For any $(m_1,m_2)\in \NN^2$ satisfying (\ref{exponent}), we write $m_2=am+j'$ with $0\leq j'< a$. We notice that the integer $j'$ depends only on $q$, hence on $j$ because $q=b(am+j')-am_1$ so that $a$ divides $bj'-q$. For the rest of the proof we denote it $j'=p(j)$. By a straightforward calculation the general exponent of $x_1$ in $\restr{\frac{\partial^r \varphi_{A',\beta}^{(j)}}{\partial x_0^r}}{x_0=0}$ is
$\frac{\beta -j}{ka}-m_1 = \frac{\beta -r}{ka} -\frac{bp(j)}{a}-bm$
%$$\frac{\beta -j}{ka}-m_1 = \frac{\beta -kq-r}{ka}-m_1 = \frac{\beta -r}{ka} -\frac{q}{a}-m_1 = $$ $$ \frac{\beta -r}{ka} -\frac{b(am+p(j))-am_1}{a}-m_1 = \frac{\beta -r}{ka} -\frac{bp(j)}{a}-bm$$
which equals the general exponent of $x_1$ in $\psi_{A,\beta'}^{(p(j))}$ by (\ref{def-Gamma-series-a-b}) applied to $\beta'= \frac{\beta -r}k$. The series $\varphi_{A',\beta}^{(j)}$ and $\psi_{A,\beta'}^{(p(j))}$ are indexed respectively by $(m_1,m_2)$ and $m$. Their terms are in a 1-1 correspondence through  the relations
\[
m_1=bm+\frac{bp(j)-q}a  \,\text{ and }\, m_2=am+p(j).
\]

Now we compare the corresponding coefficients in
%$\restr{\frac{1}{r!}\frac{\partial^r \varphi_{A',\beta}^{(j)}}{\partial x_0^r}}{x_0=0}$
$\restr{\frac{\partial^r \varphi_{A',\beta}^{(j)}}{\partial x_0^r}}{x_0=0}$ and in  $\psi_{A,\beta'}^{(p(j))}$.
The quotient of these two coefficients is well defined, for all $m\geq 0$, when $\beta \notin \NN$ and it is
\begin{equation}\label{quotients}\left(\frac{[\frac{\beta -kq -r}{ka}]_{m_1}(kq+r)!}{m_2!}\right)\left(\frac{[\frac{\beta'-p(j)b}{a}]_{bm}p(j)!}{(am+p(j))!}\right)^{-1}= \frac{\Gamma(z-bm+m_1)(kq+r)!}{\Gamma(z)p(j)!}
\end{equation}  where $z-1=\frac{\beta'-p(j)b}{a}= \frac{\beta-r}{ka} -\frac{p(j)b}{a}$.

This shows that this quotient is independent of the term chosen, because the integer $bm-m_1=\frac{q-bp(j)}a$ does not depend on $m,m_1$.
%Furthermore when $j=kq+r$ is fixed the relation (\ref{exponent}) reads also as $q+a(m_1-bm)-pb=0$ which determines a unique $p=p(j)$ independently of $m,m_1$.
%\paco{Say something about the linear combination of the $r$-th derivative of $\varphi^{(j)}_{A',\beta}$ in terms of %$\varphi^{(p)}_{A,\frac{\beta -r}{k}}$.}
So, if $\beta \not\in \NN$ we have found a constant $\lambda_j\in\CC^*$ such that $$\restr{\frac{\partial^r \varphi_{A',\beta}^{(j)}}{\partial x_0^r}}{x_0=0} = \lambda_j \psi^{(p(j))}_{A,\frac{\beta-r}{k}}\quad \mbox{ and }\quad \restr{\frac{\partial^{\ell} \varphi_{A',\beta}^{(j)}}{\partial x_0^{\ell}}}{x_0=0}=0 \quad  \mbox{ if } \ell \neq r.$$

When $\beta \in\NN$, we write $\beta=j_0+kah=r_0+kq_0+kah$, with unique  $0\leq r_0< k$ and $0\leq q_0< a$. If $j\not= j_0=r_0+kq_0$ the quotient in (\ref{quotients}) is still well defined for all $m$ and the relation $\lambda_j\not= 0$ still valid.

If $\beta\in r_0+k(\NN \setminus \NN (a,b))$, the quotient is still well defined but we have $\lambda_{j_0}= 0$  by Remark \ref{basis-phi(0,x)-special-case} since
$\varphi_{A',\beta}^{(j_0)}$ is a polynomial solution of $\cM_{A'}(\beta)$ and $\cM_{A}\left(\frac{\beta-r_0}{k}\right)$ has no non zero polynomial solution for $\frac{\beta-r_0}{k} \notin \NN(a,b)$.

%Assume first $\beta' = \frac{\beta - r_0}{k} = ha+ q_0\in \NN \setminus \NN A$. Then $ha+q_0= (h-n_1)a+n_2b$ with necessarily $h<n_1$. So $\left[ \frac{\beta - kq_0 - r_0}{ka}\right]_{n_1}= [h]_{n_1}=0$. This proves that the coefficient $\lambda_{j_0}$ is zero.

%another way to see that is that  $\varphi_{A',\beta}^{(j_0)}$ is a polynomial while $\varphi_{A,\beta'}^{p(j)}$ is not. So $\lambda_{j_0}=0$.

Assume now $\frac{\beta - r_0}{k}= ha+q_0\in \NN (a,b)$ so that  $ha+q_0 = ua+vb$ for some $u,v\in \NN$. We can write in a unique way $q_0 = -n_1 a + n_2 b$ with $0\leq n_2< a$. One has $(h-n_1) a + n_2b = ua + vb$ which forces $v\geq n_2$ and $h-n_1- u \geq 0$. In particular, $h\geq n_1$ and $[h]_{n_1} \neq 0$ and so $\lambda_{j_0}\not= 0$. This proves 1) and 2) taking $j'_0:=p(j_0)$.
\end{proof}

%if and only if $j=j_0$.  and $\restr{\frac{\partial^{r_0} \varphi_{A',\beta}^{(j_0)}}{\partial x_0^{r_0}}}{x_0=0}$ is a non zero polynomial if and only if $\beta'=\frac{\beta-r_0}{k}\in \NN A$ or $\beta \in r_0 + k (\NN \setminus \NN A)$ according to Remark xxx  \ref{basis-phi(0,x)-special-case}.

%In particular each  $\frac{\partial_0^{r}\varphi_{A',\beta}^{(j)}}{\partial x_0^{r}}(0,x_1,x_2)$ is a solution of the hypergeometric system $\cM_{A,\frac{\beta -r}{k}}$ that can be verify directly using the generators of the hypergeometric ideal
%$H_A(\frac{\beta -r}{k})$.

%$\varphi^{(j)}_{A',\beta}$ to the vector

%superfluous
%\begin{remark} We will see in Theorem \ref{a.e.are-basis-a1-an} how to describe a base of the stalk of solution space $$\bigoplus_{r=0}^{k-1} \cH om_{\cD_{X}}\left(\cM_{A}\left(\frac{\beta-r}{k}\right), {\cO_{\widehat{X\vert Y}}}(s)\right)$$ at a point in $Y \setminus Z$, for all $\beta\in \CC$ and $s\geq \frac{b}{a}$.
%\end{remark}

%$a\xrightarrow[\alpha]{\epsilon}b$ $a\xrightarrow[\phantom{xxxxxxxxxxxxxxxxxxxxxxxxxx}]{\epsilon}b$

\section{Hypergeometric integral of exponential type}\label{Section-hyper-geom-int-exp-type}
%We introduce new variables $t=(t_1,\ldots,t_d)$. For any $b=(b_1,\ldots,b_d)\in \ZZ^d$ we denote by $t^b$
%the Laurent monomial
%$t_1^{b_1}\cdots t_d^{b_d}$.
With the pair $(A,\beta)$ one associates the following integral (called hypergeometric integral of exponential type):
$$I_\gamma(A,\beta;x) = I(\beta; x):= \int_{\gamma} t^{-\beta -1} \exp\left(\sum_{j=1}^n x_j t^{a_j}\right) dt$$ where $A=(a_1,\ldots,a_n)$ and $\gamma$ is a cycle in
the rapid decay homology with closed support of M. Hien. We also write
$I_\gamma(\beta;x)=I(\beta; x)$
%if we want to emphasize the dependence of
%this integral expression also on $A$ and $\gamma$ or simply on $\gamma$
if there is no possible confusion on the matrix $A$.

The integral $I(\beta; x)$ satisfies the equality $P(I(\beta; x))=0$ for any $P$ in the hypergeometric ideal $H_A(\beta)$. So we consider $I(\beta; x)$ as a solution of $\cM_A(\beta)$.  Our goal is to give an asymptotic expansion of $I(\beta;x)$ as Gevrey series along the coordinate hyperplane  appearing in the singular support of $\cM_A(\beta)$.

We are going to prove in this case that all the Gevrey series solutions can be obtained as an asymptotic expansion of such an integral.
%describe in detail the case of a matrix with one row $A=(a_1,\dots , a_n)$, and
%We will give partial results and example in the case of multidimensional integration cycle \paco{Think about that ...}.

In the one row matrix case the rapid decay cycles are easy to describe and we prove first an asymptotic expansion statement. In the following proposition we consider a path $\gamma~:\RR\to \CC$ such that the arguments of $x_nt^{a_n}$ and of $x_{n-1}t^{a_{n-1}}$ both have limits in the open interval  $]{\frac{\pi}{2}},{\frac{3\pi}{2}}[$ when $t\to +\infty$ or $t\to -\infty$. For a fixed $\gamma$ this condition remains valid in some open sectors in the spaces $\CC^*$ for the variables $x_{n-1}$ and $x_n$. These paths are exactly the rapid decay cycles for the function  $t\to t^{-\beta -1} \exp\left(\sum_{j=1}^n x_j t^{a_j}\right)$ and for its restriction to the hyperplane $x_n=0$.

\begin{proposition}\label{expansion2}
The integral depending on $x=(x_1,\dots, x_n)$:
\[
I(\beta; x)=\int_{\gamma}t^{-\beta-1} \exp\left(\sum_{j=1}^n x_j t^{a_j}\right)dt
\]
admits an asymptotic expansion $\sum_{k=0}^\infty c_k(x_1,\dots, x_{n-1}) x_n^k$ \; for $x_n$ tending to zero whose coefficients are
\begin{equation}\label{coefficient-c-k}
c_k(x_1,\dots, x_{n-1})=\frac{1}{k!}\int_{\gamma}t^{-\beta-1+a_nk}\exp\left(x_1t^{a_1}+\dots +x_{n-1}t^{a_{n-1}}\right) dt.
\end{equation}
For a fixed cycle $\gamma$ this expansion is valid  in a product $\CC^{n-2}\times S_{n-1}\times S_n$ involving open sectors $S_i$ in $\CC^*$.
\end{proposition}
\begin{proof}
By developing the exponential $e^{x_nt^{a_n}}$ we may write
\[
I(\beta; x)=\int_{\gamma}t^{-\beta-1}\exp\left(x_1t^{a_1}+\dots +x_{n-1}t^{a_{n-1}}\right)\left(\sum_{k=0}^{\infty} \frac{(x_nt^{a_n})^k}{k!} \right) dt
\]
and we get
\begin{equation}\label{K_N} I(\beta; x_1,\dots, x_n)=\sum_{k=0}^N
\frac{x_n^k}{k!} \int_{\gamma}  t^{a_n k-\beta -1} \exp\left(x_1t^{a_1}+\dots +x_{n-1}t^{a_{n-1}}\right)
dt + R_N(\beta; x_1,\ldots,x_n)
\end{equation}
where $$R_N(\beta; x_1,\ldots,x_n) = R_N(\beta; x)= \int_{\gamma} t^{-\beta -1}\exp\left(x_1t^{a_1}+\dots +x_{n-1}t^{a_{n-1}}\right)
\left(\sum_{k=N+1}^\infty \frac{(x_nt^{a_n})^k}{k!} \right)dt.$$

These integrals are all convergent on $\gamma$ due to its rapid decay properties and we shall use the following elementary lemma involving an auxiliary complex variable  $z\in \CC$~:
\begin{lemma} Let $r_N(z)=\sum_{k=N+1}^\infty \frac{z^k}{k!}$ be the remainder of order $N$ of the exponential power series.
There exists a positive
real number $C'_N$, depending only on $N$, such that  for
all $z$ with $\Re z < 0$ one has
\[
|r_N(z)|\leq C'_N|z|^{N+1}.
\]
\end{lemma}
Since there is a compact set $K$ such that for $t\in \gamma \setminus K$ we have $\Re (x_nt^{a_n})<0$   there is a possibly larger constant $C_N$ depending also on $\gamma$ such that
\[
\forall t\in \gamma , \; \left|r_N(x_nt^{a_n})\right|\leq C_N|x_nt^{a_n}|^{N+1}
\]
So we have
\begin{align*}
\left|R_N(\beta; x) \right|=&\left| x_n^{N+1} \int_{\gamma} t^{a_n(N+1)-\beta-1}\exp\left(x_1t^{a_1}+\dots +x_{n-1}t^{a_{n-1}}\right) {r_N(x_nt^{a_n})\over (x_nt^{a_n})^{N+1}} dt\right|\\ = &|Q_N(\beta; x_1,\dots, x_{n})|\cdot |x_n|^{N+1}.
\end{align*}
This proves the existence of an asymptotic expansion which is locally uniform with respect to $(x_1,\dots, x_{n-1})$. Indeed as indicated in the statement, the domain of convergence of the last integral $Q_N$ contains the product of $\CC^{n-2}$ by a product of sectors in the variables $x_{n-1},x_n$ . It is convergent since the integrand is bounded by the integrable function~:
$$C_N\left|t^{a_n(N+1)-\beta-1}\exp\left(x_1t^{a_1}+\dots +x_{n-1}t^{a_{n-1}}\right) \right|.$$
\end{proof}

\section{Gevrey expansions of hypergeometric integrals for $A=(a_1,\dots , a_n)$}
First of all we prove the following
\begin{proposition}\label{a-e-is-Gevrey}
The asymptotic expansion of the integral depending on $(x_1,\dots, x_n)$:
\[
I(\beta; x)=\int_{\gamma}t^{-\beta-1} \exp\left(\sum_{j=1}^n x_j t^{a_j}\right)dt
\]
given in Proposition \ref{expansion2} is a Gevrey series of order less than or equal to  $s=\frac{a_n}{a_{n-1}}$ with respect to $x_n=0$.
\end{proposition}
\begin{proof}
We set $(a_{n-1},a_n)=(da,db)$ with $\gcd(a,b)=1$ and for each $k\in\NN$, $k=aq+j$ with $0\leq j < a$. Looking at the exponents of $t$ in the integrands, the coefficients $c_k$ in (\ref{coefficient-c-k})
satisfy by derivation under the sign $\int_{\gamma}$ the relation
\[
(qa+j)!c_{qa+j}(x_1,\dots, x_{n-1})=j!{\del^{qb}c_j\over\del{x_{n-1}}^{qb}}(x_1,\dots, x_{n-1})
\]
Since each of the functions $c_0,\dots,c_{a-1}$ is holomorphic, we have in a small enough neighbourhood of a point $x_1,\dots ,x_{n-1}$ with $\Re x_{n-1}<0$ a uniform upper bound involving a constant $K$ that we can choose common to all the indices $j=0,\ldots,a-1$ :
\[
(qa+j)!\left|c_{qa+j}(x_1,\dots, x_{n-1})\right|\leq j!(qb)!K^{qb}
\]
%\left|c_j(x_1,\dots, x_{n-1})\right|
Then a local upper bound of %$c_k(x_1,\dots, x_{n-1})/(k!)^{\frac ba}$
$c_k(x_1,\dots, x_{n-1})$ is the quotient $$(a-1)!\frac{(\lfloor\frac{kb}a\rfloor)!K^{\frac{kb}{a}}}{k!}.$$ Hence, the asymptotic expansion that we consider is a  Gevrey series of order at most $\frac ba = \frac{a_n}{a_{n-1}}$. \end{proof}

\begin{remark} We will see, as a consequence of Theorem \ref{a.e.are-basis-a1-an}, that in fact these Gevrey series have index equal to $\frac{a_n}{a_{n-1}}$ when $\beta$ is generic enough. See Section \ref{gevrey-mod-convergent}.
\end{remark}

\subsection{Case $A=(a,b)$}\label{casab} The aim of this Subsection is to compare, in the case $A=(a,b)$,  the Gevrey asymptotic expansions of the hypergeometric integrals to the Gevrey solutions described in Proposition \ref{basis-gevrey-solutions-dim-2}.  This comparison is proved in Theorem \ref{a.e.are-basis-in-dim2}. We consider here $a,b\in \ZZ$,  $1\leq a<b$ and $a,b$ are relative primes.

%From \cite{FC2} (see Proposition \ref{}) a basis of the Gevrey solutions --with respect to
%$x_2=0$-- of $M_A(\beta)$ is, for $\beta$ generic, equal to
%$\{\varphi^{(0)},\dots, \varphi^{(a-1)}\}$ where
%\begin{equation}\label{1}
%\varphi^{(j)} = x_1^{\frac{\beta-jb}{a}} \sum_{m\geq 0}
%\frac{[\frac{\beta-jb}{a}]_{bm}}{[am+j]_{am} } x_1^{-bm}x_2^{am+j}
%\end{equation} \paco{this formula has been already written at least two times ...}

We consider the integrals
\[
I_\gamma (A,\beta;x) = I_\gamma(\beta; x) = \int_{\gamma} t^{-\beta -1} \exp\left(x_1 t^a + x_2 t^b\right) dt.
\] with respect to various specific cycles of rapid decay $\gamma=C_p$ as in Figure 1.

%\centerline{$\stackrel{\includegraphics[scale=0.50]{chemin.pdf}}{fig.1}$}

\begin{center}
\begin{picture}(0,0)%
\includegraphics[scale=0.60]{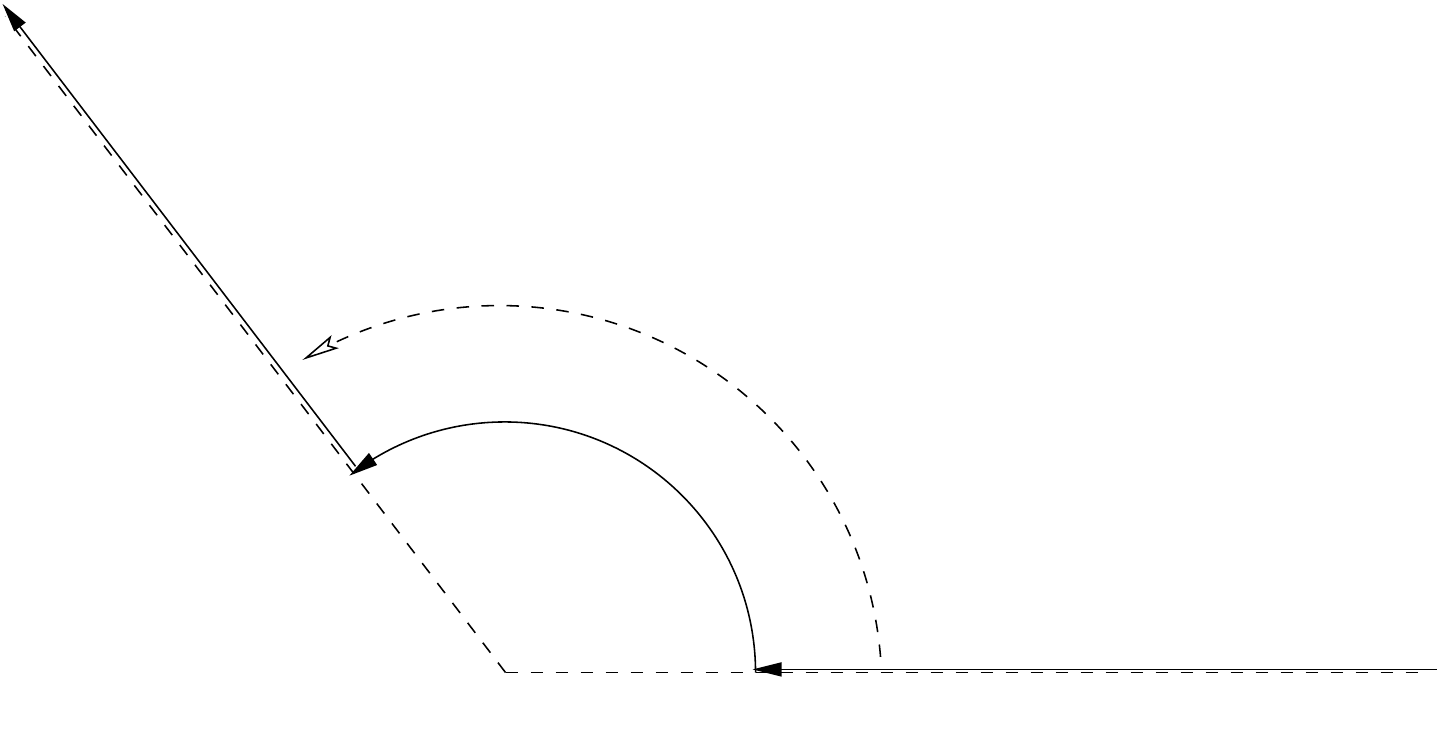}%
\end{picture}%
\setlength{\unitlength}{2368sp}%
\begingroup\makeatletter\ifx\SetFigFont\undefined%
\gdef\SetFigFont#1#2#3#4#5{%
  \reset@font\fontsize{#1}{#2pt}%
  \fontfamily{#3}\fontseries{#4}\fontshape{#5}%
  \selectfont}%
\fi\endgroup%
\begin{picture}(6909,3511)(3574,-4250)
\put(7126,-4186){\makebox(0,0)[lb]{\smash{{\SetFigFont{8}{14.4}{\rmdefault}{\mddefault}{\updefault}{\color[rgb]{0,0,0}$\epsilon$}%
}}}}
\put(5026,-3286){\makebox(0,0)[lb]{\smash{{\SetFigFont{8}{14.4}{\familydefault}{\mddefault}{\updefault}{\color[rgb]{0,0,0}$\epsilon\omega_p$}%
}}}}
\put(6226,-2161)
{\makebox(0,0)[lb]{\smash{{\SetFigFont{8}{14.4}{\familydefault}{\mddefault}{\updefault}{\color[rgb]{0,0,0}${2 \pi p \over b}$}%
}}}}
\put(6200,-4500)
{\makebox(0,0)[lb]{\smash{{\SetFigFont{8}{14.4}{\familydefault}{\mddefault}{\updefault}{\color[rgb]{0,0,0}{${\rm Fig.}\, 1.$ The cycle $C_p$.}}%
}}}}
\end{picture}%

\phantom{xx}
\end{center}

This path consists of the real half line $[\epsilon ,+\infty[$ negatively oriented, and the half line $\omega_p\cdot[\epsilon ,+\infty[$ where $\omega_p=e^{2i\pi p\over b}$ is a $b$--th root of unity joined by the arc of a circle $ \{ \epsilon e^{i\theta}\mid 0\leq \theta \leq {2p\pi\over b}\}$  with $\epsilon >0$.

The cycles $C_p\otimes \sigma$, for $1\leq p \leq b$ are a basis of the rapid decay homology with closed support as defined in \cite{Hien07} where $\sigma$ is a section of the local system $\CC\cdot t^{-\beta-1}\exp(x_1t^a+x_2t^b)$ in $\CC^{\ast}$. We choose the determination of $\log t$ as being real on $[\epsilon,+\infty[$.
When $\beta$ is generic, in practise here when $\beta \notin \ZZ$, they are all non compact and when $\beta$ is an integer the last cycle $C_b$ is compact being equivalent to the circle of radius $\epsilon$ because $t^{-\beta-1}$ is uniform on $\CC^*$.
The integral along $C_p$ does not depend on	 the choice of $\epsilon$ and we have an asymptotic expansion which is just a particular case of Proposition \ref{expansion2}:

\begin{proposition}\label{expansion}
The integral depending on $(x_1,x_2)$:
\[
I_{C_p}(\beta; x_1,x_2)=\int_{C_p}t^{-\beta-1}e^{x_1t^a +x_2t^b }dt
\]
is defined when $\Re x_2<0$ and admit an asymptotic expansion $\sum_{k=0}^\infty c_k(x_1) x_2^k$ whose coefficients are
\[
c_k(x_1)=\frac{1}{k!}\int_{C_p}t^{-\beta-1+bk}e^{x_1t^a}dt.
\]
This expansion is valid in the open sector defined by $\Re x_1<0,\Re(\omega_p ^ax_1)<0$.
\end{proposition}
\qed

\begin{remark}\label{sectors}
Notice that the open sector of the variable $x_1$ described in Proposition \ref{expansion} is empty when $\omega_p^a=-1$ which happens if $(2k+1)b=2pa$. In Proposition \ref{DAsecteurs} we will show that if we restrict the domain of the variable $x_2$ to some sector $\pi-\epsilon<\arg x_2<\pi+\epsilon$ for a small $\epsilon>0$ we may enlarge the domain of validity in the variable $x_1$.
\end{remark}

We give now a more precise description of the asymptotic expansion of the integral $I_{C_p}(\beta; x_1,x_2)$.
The coefficient $c_k(x_1)$ can be decomposed as a sum $-I_{k,1} + J_k + I_{k,2}$ where
$$ I_{k,1}(\epsilon,x_1)=\frac{1}{k!}\int_\epsilon^{+\infty} s^{bk-\beta-1}e^{x_1s^a}ds$$
$$ J_k(\epsilon,x_1)=\frac{1}{k!}\int_0^{2p\pi\over b} \epsilon^{bk-\beta -1 } e^{-i\theta (\beta +1-bk)} e^{x_1 \epsilon^a e^{ia\theta } } i\epsilon \, e^{i \theta} d\theta $$
$$ I_{k,2}(\epsilon,x_1)=\frac{1}{k!}\int_{\omega_p\cdot[\epsilon,+\infty[} t^{-\beta-1+bk}e^{x_1t^a}dt=\frac{1}{k!}\int_{\epsilon}^{+\infty}e^{{2ip\pi\over b}(bk-\beta)} s^{bk-\beta -1} e^{x_1\omega_p^as^a} ds.$$

For $k$ large enough $\Re (bk-\beta)>0$ and the limit of $J_k(\epsilon,x_1)$ is 0 when $\epsilon$ tends to 0. Under the same condition the limit of the sum $I_{k,1}$ and of $I_{k,2} $ exist  and then
\begin{equation}\label{coef}
 c_k(x_1)=\frac{1}{k!}\int_0^{+\infty}e^{{2ip\pi\over b}(bk-\beta)} s^{bk-\beta -1} e^{x_1\omega_p^as^a} ds-\frac{1}{k!}\int_0^{+\infty} s^{bk-\beta-1}e^{x_1s^a}ds.
\end{equation}
Let us denote $c_k(x_1)=I_{k,2}(x_1)-I_{k,1}(x_1)$ this decomposition of the coefficient $c_k(x_1) $ as a difference. We transform $I_{k,1}(x_1)$ for $x_1$ real negative by  the change of variable
$u=|x_1|s^a$, $ds={sdu\over au}$.

\begin{equation}\label{1000}
I_{k,1}(x_1)=\frac{1}{k!} \int_0^{+\infty} s^{bk-\beta-1} e^{x_1 s^a} ds ={1\over k!}\int_0^{+\infty}\left({u\over|x_1|}\right)^{bk-\beta\over
a}{e^{-u}du\over au}.
\end{equation}
The final result is
\[
I_{k,1}(x_1)
=\frac{1}{a\, k!}e^{-i\pi({\beta-kb\over
a})}x_1^\frac{\beta-kb}{a} \int_0^{+\infty}u^{{bk-\beta\over a}-1}e^{-u}du
\]
when we choose $\pi$ as a determination of the argument of $x_1$. The equality is valid on the half plane $\Re x_1<0$ because both sides are holomorphic and coincide on the real negative axis by the equation (\ref{1000}). The last integral equals $\Gamma({bk-\beta\over a})$.

Similarly the first part of the equation (\ref{coef}) for $c_k(x_1)$ can be evaluated for the values of $x_1\in \CC$ such that $x_1\omega_p^a$ is real negative and setting $u=|x_1|s^a$ for such a fixed $x_1$
\[
I_{k,2}(x_1)
=\frac{1}{k!}\int_0^{+\infty}e^{{2ip\pi\over b}(bk-\beta)} s^{bk-\beta -1} e^{x_1\omega_p^as^a}ds={1\over k!}\int_0^{+\infty}e^{{2ip\pi\over b}(bk-\beta)}\left({u\over|x_1|}\right)^{bk-\beta\over
a}{e^{-u}du\over au}
\]
The appropriate argument for $\omega_p^a$ is ${2ap\pi\over b}$, therefore in order to get $\arg x_1\omega_p^a=\pi$ one must set $\arg x_1=\pi-{2ap\pi\over b}$. By a calculation identical to the one used for $I_{k,1}(x_1)$ we get:
\[
|x_1|^{\beta-kb\over
a}=x_1^{\beta-kb\over
a}e^{-i\pi({\beta-kb\over
a})}e^{{2iap\pi\over b}({\beta-kb\over
a})}=x_1^{\beta-kb\over a}e^{-i\pi({\beta-kb\over
a})}e^{{2ip\pi\over b}\beta}
\]
\[
I_{k,2}(x_1)
={1\over a\, k!}x_1^{\beta-kb\over a}e^{-i\pi({\beta-kb\over
a})}\int_0^{+\infty}u^{{bk-\beta\over
a}-1}{e^{-u}du}.
\]

We get an expression of $I_{k,2}(x_1)$ formally identical to the one for $I_{k,1}(x_1)$. However the respective determinations of the argument of $x_1$ that we use in these two integrals are in general different in the common domain of definition. This domain is the intersection $S_p$ of the two half planes $\Re x_1<0$ and $\Re x_1\omega_p ^a<0$. Therefore in this common sector  there is a constant $c(k,p)$ such that $I_{k,1}(x_1)=c(k,p)I_{k,2}(x_1)$.

% More precisely define $f(x_1)$, and  $f_p(x_1)$ to be the determinations of  $x_1^{\beta-kb\over a}$ which appear in $B_k(x_1)$ and $ A_{,k}(x_1)$  respectively.

% For $f$ and in a neighbourhood of $-1=e^{i\pi}$ we have
% \[
% f(re^{i(\pi+\theta)})=r^{\beta-kb\over a}e^{(i\pi+\theta)({\beta-kb\over a})})
% \]
% and for $f_p$ in a neighbourhood of $-\omega ^{-a}=e^{i(\pi-\frac{2\pi ap} {b})}$
% \[
% f_p(re^{i(\pi-\frac{2\pi ap}{b}+\theta)})=r^{\beta-kb\over a}e^{(i\pi+\theta)% ({\beta-kb\over a})}\cdot e^{i\frac{-2\pi p(\beta-kb)}{b}}
% =r^{\beta-kb\over a}e^{(i\pi+\theta)({\beta-kb\over a})}\cdot  e^{-i\frac{2\pi % p\beta}{b}}
% \]
Assume that $\omega_p^a\neq -1$. We have $\arg(\omega_p^{-a})=-\frac{2\pi ap}{b}$ and we define the integer $\ell\in \{0,\dots , a\}$ by the property that $\alpha=2\pi(\ell-\frac{ ap}{b})$ is the determination  of $\arg(\omega_p^{-a})$ in $]-\pi,\pi[$
\begin{equation}\label{tours}
-\pi<2\pi\ell-\frac{2\pi ap}{b}<\pi.
\end{equation}
In other words $\ell=\lceil\frac{ap}{b}-\frac12\rceil$, and since $\omega_p^a\neq -1$,   $\ell$ is the unique integer such that $|\frac{ ap}{b}-\ell|<\frac12$. The argument used for $x_1$ in  the sector of validity of $I_{k,2}(x_1)$  is centered on the value $\pi-{2ap\pi\over b}\in ]-2\pi\ell,-2\pi(\ell-1)[.$
%\pacoverde{To check this last interval, c'est check\'e!}. *
Therefore it is equal to $\arg x_1-2\pi\ell$ if we denote by $\arg x_1$ the standard determination in $]0,2\pi[$ used for $I_{k,1}(x_1)$. By treating the effect of this difference on the monomial $x_1^{\beta-kb\over a}$ we obtain in the sector $S_p$~:
\begin{align}\label{1001}
I_{k,2}(x_1)=I_{k,1}(x_1)\times e^{-2i\pi\ell({\beta-kb\over a})}
\end{align}
If we set $k=am+j$ it results in
\[
I_{am+j,2}(x_1)=I_{am+j,1}(x_1)\cdot \left(e^{2i\pi{jb-\beta\over a}}\right)^{\ell}.
\]
The asymptotic expansion of $I_{C_p}(x_1,x_2)$ is valid in the same sector by the proof of Proposition \ref{expansion2} and we obtain still limiting ourselves to $k$
or $m$ large enough :
\begin{equation}\label{coefficient}
\begin{multlined}
 c_{am+j}(x_1)=
\left(\left(e^{2i\pi{jb-\beta\over a}}\right)^{\ell}-1\right)I_{am+j,1}(x_1)
\\
=\frac{\left(e^{2i\pi{jb-\beta\over a}}\right)^{\ell}-1}{a(am+j)!}e^{i\pi({jb-\beta\over
a})}e^{i\pi bm} x_1^{{\beta-jb\over
a}-bm}\Gamma\left(bm+{jb-\beta\over a}\right)
\end{multlined}
\end{equation}

 Assume that $\beta\notin \ZZ$. Recall from (\ref{def-Gamma-series-a-b}) and  Proposition \ref{basis-gevrey-solutions-dim-2}, the elements  $\psi^{(j)}:=\psi^{(j)}_{(a,b),\beta}$ of a basis of the Gevrey solution space of order less than or equal to $s\geq \frac{b}{a}$ at generic points in $Y=(x_2=0)$. We want to arrange the expression of $c_{am+j}(x_1)$ in (\ref{coefficient}) so as to recognize in it a multiple by a factor depending only on $j$ of the coefficient of $x_2^{am+j}$ in the expansion of $\psi^{(j)}$. See  (\ref{coefficient2}) below.

  Using the well known identity
\[
\Gamma(z)\Gamma(1-z)={\pi\over \sin \pi z} \text{ and } {\Gamma(z+1)\over \Gamma(z-m+1)}=[z]_m=z(z-1)\dots (z-m+1)
\]
\[
\Gamma\left(bm+{bj-\beta\over a}\right)={(-1)^{bm}\pi\over \sin(\pi\cdot {bj-\beta\over a} )\Gamma({\beta-bj\over a}-bm+1)}={(-1)^{bm}\pi\over \sin(\pi\cdot {bj-\beta\over a} )\Gamma({\beta-bj\over a}+1)}\cdot \left[{\beta-bj\over a}\right]_{bm},
\]
the expression of the function $c_{am+j}$ is finally~:
\begin{equation}\label{coefficient2}
c_{am+j}(x_1)=\left[\left(e^{2i\pi{jb-\beta\over a}}\right)^{\ell}-1\right]{\pi e^{-i\pi({\beta-jb\over a})}\over a\,j!\sin(\pi\cdot {bj-\beta\over a} )\Gamma({\beta-bj\over a}+1)}\cdot
{\left[{\beta-bj\over a}\right]_{bm}\over [am+j]_{am}}x_1^{\frac{\beta-bj}{a}-bm}
\end{equation}

Now we are ready to prove a more precise statement about asymptotic expansion of the integral along the cycle $C_p$.
\begin{proposition}\label{DAsecteurs}
\begin{enumerate}
\item The integral $I_{C_p}(x_1,x_2)$ has, provided that $e^{2i\pi pa\over b}\neq -1$, an asymptotic expansion which  is a linear combination of the Gevrey series $\psi^{(j)}$. When $\beta \notin \ZZ$ the coefficient of $\psi^{(j)}$ in this linear combination is equal to the product of $q_j^{\ell}-1:=\left(e^{2i\pi{jb-\beta\over a}}\right)^{\ell}-1$ by a non zero constant $\lambda_j$ which does not depend on $p$ while $|\ell-\frac{ap}{b}|<\frac12.$ When $\beta\in \ZZ$ the same is true if $\frac{\beta-bj}a \notin \ZZ$. Finally when $\beta\in \ZZ\setminus \NN(a, b)$ the coefficient of $\psi^{(j_0)}$ is zero for the unique $j_0 $ such that $\frac{\beta-bj_0}a \in \ZZ_{<0}$.
\item Assume $\omega_p^a\neq -1$ and set $\alpha=2\pi(\ell-{ap\over b})$, $|\alpha|<\pi$. Then the expansion in (1) is valid for $\Re x_2<0$ and $x_1$ in the sector of angular width $\pi-|\alpha|$ defined by the condition :
\[\arg x_1\in\begin{cases}
 \quad ]\frac\pi2+\alpha,\frac{3\pi}2[ \text{ if }\alpha \geq 0 \\
 \quad  ]\frac\pi2,\frac{3\pi}2+\alpha[ \text{ if }\alpha \leq 0.
 \end{cases}
 \]
\item If we restrict the domain for the variable $x_2$ to a sector $|\arg (x_2)-\pi|<\epsilon$ with $\epsilon$ sufficiently small we can extend the domain of validity with respect to the variable $x_1$ to a larger sector, in such a way that for each $\ell \in \{1,\dots, a\}$ there exists  $p\in \{1,\ldots,b\}$  for which $|\ell-\frac{ap}{b}|<\frac12$ and the open extended sector contains the real negative axis.
\end{enumerate}
\end{proposition}
\begin{proof}
By Proposition \ref{basis-gevrey-solutions-dim-2} and since the integral $I_{C_p}(x_1,x_2)$ is a solution of the hypergeometric system $\cM_A(\beta)$, the asymptotic expansion that we found in Proposition \ref{expansion} is a linear combination of the Gevrey series  $\psi^{(j)}$, described in equation (\ref{def-Gamma-series-a-b}). Let us call $\mu_j$ the coefficient of this linear combination. Since the set of exponents of the series $\psi^{(j)}$ are mutually disjoint the sum $\sum_{m\geq 0}c_{am+j}(x_1)x_2^{am+j}$ must be the multiple  $\mu_j\psi^{(j)}$ of the series $\psi^{(j)}$. Assuming first that $\beta \notin \ZZ$ the calculation for $m$ large enough in equation (\ref{coefficient}) is sufficient to determine $\mu_j$ and comparing formula (\ref{coefficient2}) with the expression of the series $\psi^{(j)}$ gives the result with the value $\lambda_j=\pi e^{-i\pi({\beta-jb\over a})} \left( a\, j!\sin(\pi\cdot {bj-\beta\over a} )\Gamma({\beta-bj\over a}+1)\right)^{-1}$.

When $\beta\in\ZZ$ there is a unique $j_0\in \{0,\dots,a-1\}$ such that $\frac{\beta-bj_0}a$ is an integer. For $j\neq j_0$, the same argument applies for the determination of $\mu _j$. %and again $\mu_j\neq 0$.
Let us write $\beta=bj_0+aq$, with $q\in \ZZ$. In that case $c_{am+j_0}(x_1)=0$ for $m$ big enough. When $q<0$ or equivalently $\beta\in \ZZ\setminus \NN(a,b)$ the series $\psi^{(j_0)}$ has an infinite number of terms and again the argument based on equation (\ref{coefficient}) gives the announced result with $\mu_{j_0}=0$ by equation (\ref{coefficient}).  When $\beta \in \NN(a,b)$ the series $\psi^{(j_0)}$ reduces to a polynomial and equation (\ref{coefficient}) gives no information about $\mu_{j_0}$ to be reported in the statement of Proposition \ref{DAsecteurs}.

In the exceptional case $\beta \in \NN(a,b)$ we find directly that $\psi^{(j_0)}$ comes from an integral solution. Indeed by inspection of the $\Gamma$-series in (\ref{def-Gamma-series-a-b}) we find that $\varphi ^{(j_0)}$ is then the polynomial
$$ \sum_{m=0}^{\lfloor\frac qb\rfloor} \frac{[q]_{bm}}{[am+j_0]_{am} } x_1^{q-bm}x_2^{am+j_0}$$ and this is exactly the integral
\[
\frac1{2\pi i\,q!\,j_0!}I_\gamma (A,\beta;x) =\frac1{2\pi i\,q!\, j_0!}\int_{\gamma} t^{-\beta -1} \exp\left(x_1 t^a + x_2 t^b\right) dt
\] along the compact cycle $\gamma=(|t|=\epsilon)$ for $\epsilon >0$ small enough.

Consider the two half planes $\Re x_1<0$ and $\Re(x_1\omega_p^a)<0$ where the functions $I_{k,1}$ and $I_{k,2}$ are defined. Since $\alpha$ is the principal argument of $\omega_p^{-a}$, they intersect along the common sector $]\frac{\pi}2,\frac{3\pi}2[\,\, \cap \, \,]\frac{\pi}2+\alpha,\frac{3\pi}2+\alpha[$. This gives the second statement.

Furthermore for each $\ell\in\{1,\dots , a\}$ we choose $p\in\{1,\dots , b\}$ such that $|\alpha|$ the smallest possible which yields:
\[
\vert \alpha \vert =|2\pi \ell-\frac{2\pi ap}{b}|\leq \frac{\pi a}b
\]
This remark is useful for the proof of the statement (3) and shows that the restriction $\omega_p^a\neq -1$ is harmless since we can avoid it and reach any given $\ell$.

 In view of the last statement let us consider the family of paths $e^{i\theta}C_p$, for $|\theta|<\frac{\pi}{2b}-\frac{\epsilon}b$. When $x_2$ remains in the open sector $\vert\arg (x_2) -\pi\vert<\epsilon$ centered on real negative axis they are all rapid decay cycles for the function $t^{-\beta-1}\exp(x_1t^a+x_2t^b)$ and are equivalent in the rapid decay homology.
Therefore we have~:
\[
I_{C_p}(x_1,x_2)=\int_{C_p}t^{-\beta-1}\exp(x_1t^a+x_2t^b)dt=
\int_{e^{i\theta}C_p}t^{-\beta-1}\exp(x_1t^a+x_2t^b)dt.
\]
By Proposition \ref{expansion2} the Gevrey asymptotic expansion of $I_{C_p}(x_1,x_2)$ that we have obtained in statement (1) of the Proposition can also be written $\sum_{k=0}^\infty c_{\theta, k}(x_1) x_n^k$ with ~:
\[
c_{\theta, k}(x_1)=\frac1{k!}\int_{e^{i\theta} C_p}t^{bk-\beta-1}\exp(x_1t^a)dt=\frac1{k!}\int_{C_p}t^{bk-\beta-1}\exp(x_1e^{ia\theta} t^a)dt
\]
The domain of validity with respect to the variable $x_1$ of this asymptotic expansion therefore contains a sector $S_{\theta}$ which is the image by a rotation with an angle $-a\theta$ of the initial sector obtained from part (1) of the Proposition. Since we can continuously deform the paths of integration from $C_p$ to $e^{i\theta}C_p$ the domain of validity of the asymptotic expansion of $I_{C_p}$ is the union of all the sectors $S_{\theta}$. The original sector $S_0$ is centered at $\frac{\alpha}2+\pi$ and its width is $\pi-|\alpha|$. When $|\arg(x_2)-\pi|<\epsilon$ we have enlarged this width by an angle $\frac ab(\pi-2\epsilon)$. The width of the enlarged sector is~:
\[
\pi-|\alpha|+\frac ab(\pi-2\epsilon)>\pi-2\epsilon.
\]
Since this sector is centered at $\frac{\alpha}2+\pi\in\,\, ]\frac{\pi}2,\frac{3\pi}2[$ this proves that it contains the real negative axis  for $\epsilon>0$ small enough.
\end{proof}
\begin{proposition}\label{a.e.-gevrey-dim2-beta-generic} Assume that $\beta\notin \ZZ$. Then for every Gevrey series $\varphi$ of order less than or equal to  $s\geq \frac{b}{a}$,  which is a solution of the hypergeometric system $\cM_A(\beta)$, there is an holomorphic solution defined in a product of sectors, which is a neighbourhood of the product of the real negative axes, and which admits $\varphi$ as an asymptotic expansion. This solution can be described as an integral of the function $t^{-\beta-1}\exp(x_1t^a+x_2t^b)$ along a rapid decay cycle.
\end{proposition}
\begin{proof}
According to %the subsection \ref{casab}
Proposition \ref{basis-gevrey-solutions-dim-2} it is sufficient to prove that each of the series $\psi^{(j)}$ or what amount to the same each series $\lambda _j\psi^{(j)}$ is the asymptotic expansion of such an integral. In Proposition \ref{DAsecteurs}, we have described such integrals and asymptotic expansions as linear combinations of the  $\lambda _j\psi^{(j)}$. The number $a$ of these integrals is equal to the dimension of the space of Gevrey solutions. Therefore in order to prove the statement we just have to show that the square matrix of the coefficients of these linear combinations is invertible.  In the notations of Proposition \ref{DAsecteurs} this matrix is:
\[
M=\left(\begin{array}{cccc}
q_0-1 & q_0^2-1 & \ldots & q_0^a-1 \\
\ldots & \ldots & \ldots & \ldots \\
\ldots & \ldots & \ldots & \ldots \\
q_{a-1}-1 & q_{a-1}^2-1 & \ldots & q_{a-1}^a-1
\end{array}\right)
\]
and one show by elementary calculations \[
\det M=\left|\begin{array}{ccccc}
1& 1 & 1 & \ldots & 1\\
1& q_0 & q_0^2 & \ldots & q_0^a \\
\vdots& \ldots & \ldots & \ldots & \ldots \\
\vdots& \ldots & \ldots & \ldots & \ldots \\
1& q_{a-1}& q_{a-1}^2 & \ldots & q_{a-1}^a
\end{array}\right|=\prod_{i=0}^{a-1}(q_i-1) \prod_{i<j}(q_j-q_i)\neq0
\]
where the last inequality follows from $q_j=e^{2i\pi{jb-\beta\over a}}$, with $\beta\notin \ZZ$ and $a,b$ co-prime.
\end{proof}
\begin{theorem} \label{a.e.are-basis-in-dim2} For any $\beta\in \CC$ and for every Gevrey series $\varphi\in \cO_{\widehat{X|Y}}(s)$, $s\geq \frac{b}{a}$, which is a solution of the hypergeometric system $\cM_A(\beta)$, there is an holomorphic solution defined in a product of sectors, which is a neighbourhood of the product of the real negative axes, and which admits $\varphi$ as an asymptotic expansion. All these solutions can be described as an integral of the function $t^{-\beta-1}\exp(x_1t^a+x_2t^b)$ along a rapid decay cycle when $\beta\notin \ZZ\setminus \NN (a,b)$. When $\beta \in \ZZ \setminus \NN(a,b)$, the above integral solutions span a codimension one subspace and there is a one dimensional supplementary space obtained by expanding an integral along $[0,+\infty[$.
\end{theorem}

\begin{proof}
When $\beta \notin \ZZ$, Proposition \ref{a.e.-gevrey-dim2-beta-generic} gives a complete proof of the statement. When $\beta \in \ZZ$ we write $\beta=j_0b+aq$, $0\leq j_0<a$. The proof of the same Proposition solves the case of the Gevrey series $\psi^{(j)}$ for $j\neq j_0$. Indeed the matrix $M$ in this proposition has its last column equal to zero and the row corresponding to $\psi^{(j_0)}$ is zero too. We find that the matrix $M$ is of rank exactly $a-1$ so that all  $\psi^{(j)}$ for $j\neq j_0$ are obtained as integrals on rapid decay cycles. The non obtained Gevrey series is:
\[
\psi^{(j_0)} = \sum_{m\geq 0}
\frac{[q]_{bm}}{[am+j_0]_{am} } x_1^{q-bm}x_2^{am+j_0}.
\]
It is a polynomial if and only if $q\geq 0$, that is when $\beta\in \NN(a,b)$. We notice that it is exactly the case where the integral along a circle of radius $\epsilon >0$ centered at the origin which is equal to
\[
\int_{C_b}t^{-\beta-1}e^{x_1t^a+x_2t^b}dt=2\pi i\sum_{\ell_1 a+\ell_2b=\beta}{x_1^{\ell_1}x_2^{\ell_2}\over\ell_1!\ell_2!}
\]
is non zero. Since this is a polynomial solution we are done in this case.

Finally when $\beta \in \ZZ\setminus \NN (a,b)$, that is when $q<0$,  we notice that there is an integral holomorphic solution given by the formula:
\begin{equation} \label{jota-beta-a-b}
J_{\beta}(x_1,x_2)=\int_0^{+\infty}t^{-\beta-1}\left(e^{x_1t^a+x_2t^b}-P_{\beta}(x_1,x_2,t)\right)dt
\end{equation}
where  $P_{\beta}$ is zero if $\beta <0$, and otherwise is the Taylor polynomial of degree $\leq\beta$  for $t\to e^{x_1t^a+x_2t^b}$ :
\[
P_{\beta}(x_1,x_2,t)=\sum _{\ell_1 a+\ell_2b\leq\beta}{x_1^{\ell_1}x_2^{\ell_2}t^{\ell_1 a+\ell_2b}\over\ell_1!\ell_2!}.
\]
In fact $\deg_tP_\beta<\beta$, by the condition $\beta \notin \NN(a,b)$. This yields the convergence of $J_{\beta}$ at $+\infty$.

The fact that $J_{\beta}$ is a solution of the system $H_A(\beta)$ is completely similar to the proof in \cite[Section 2]{Adolphson} for the case of rapid decay cycles. The reason is first that since $\del_1^b-\del_2^a$ annihilates $e^{x_1t^a+x_2t^b}$, it annihilates as well the coefficients of all the monomials $t^j$ in its power expansion and  hence also the polynomial $P_{\beta}(x_1,x_2,t)$. Concerning the Euler operator $\chi=ax_1\del_1+bx_2\del_2$ we just have to notice that the relation for the integrand which leads to the proof is still valid :
\begin{equation*}  \chi\left(e^{x_1t^a+x_2t^b}-P_{\beta}(x_1,x_2,t)\right)=t\del_t\left(e^{x_1t^a+x_2t^b}-P_{\beta}(x_1,x_2,t)\right).
\end{equation*}

In order to end the proof of Theorem \ref{a.e.are-basis-in-dim2} we just have to check that $J_{\beta}$ admits under the same conditions as in Section \ref{casab} an asymptotic expansion (of Gevrey order less than or equal to $\frac{b}{a}$) which must then be a linear combination $\sum\mu_j\psi^{(j)}$ of all the $\psi^{(j)}$.
We rewrite it as follows
 \[
J_{\beta}(x_1,x_2)=\int_0^{+\infty}\left(\sum_{k=0}^{\infty}t^{-\beta-1+bk}\left(e^{x_1t^a}-\sum_{\ell a\leq\beta-kb}\frac{x_1^\ell t^{\ell a}}{\ell!}\right){x_2^k\over k!}\right)dt
 \]
and this leads to the existence of an asymptotic expansion completely similar to the one given in the general case, but with a coefficient of  $x_2^k$ given by~:
\begin{equation}  \label{c-k-jota-a-b}
 c_k(x_1)=\frac1{k!}\int_0^{+\infty}t^{-\beta-1+bk}\left(e^{x_1t^a}-\sum_{\ell a\leq\beta-kb}\frac{x_1^\ell t^{\ell a}}{\ell!}\right)dt
\end{equation}
which is again a convergent integral. The change of variable $s=-x_1t^a$ for $x_1\in \RR_{<0}$ works exactly in the same way as in Remark \ref{sectors} and Proposition \ref{DAsecteurs}. In fact as soon as $k$ is large enough the correcting term $\sum_{\ell a\leq\beta-kb}\frac{x_1^\ell t^{\ell a}}{\ell!}$ is zero and we find explicitly~:
\[
 c_k(x_1)=\frac{1}{a k!}e^{-i \pi \frac{\beta - bk}{a}}\Gamma\left({bk-\beta\over a}\right)x_1^{\beta-kb\over a}.
\]
Considering the coefficients $c_{am+j_0}(x_1)$ for $m$ large enough, we see that the coefficient $\mu_{j_0}$ of $\psi_{A,\beta}^{(j_0)}$ is non zero in the asymptotic expansion of $J_{\beta}$ as expected.
\end{proof}

\subsection{A basis of Gevrey asymptotic expansions}  %\pacoverde{problema con los puntos $p$ y los n\'{u}meros $p$ ...}
Let $A=(a_1,\ldots,a_n)$ with integers $0<a_1<\cdots < a_n$ and $\gcd(a_1,\ldots,a_n)=1$. By a rapid decay cycle we always mean in this section a rapid decay cycle for the exponentials of both polynomials $\sum_{1\leq j\leq n} x_j t^{a_j}$ and $\sum_{1\leq j\leq n-1} x_j t^{a_j}$ . In Propositions \ref{expansion2} and \ref{a-e-is-Gevrey} we prove that for any $\beta \in \CC$ and any rapid decay 1-cycle $\gamma$ there is an asymptotic expansion of the integral  $I_\gamma(A,\beta,x)$. We denote it :
$$\Phi_\gamma(A,\beta;x) = \aex(I_\gamma(A,\beta,x)) = \aex{\left(\int_\gamma t^{-\beta-1} \exp\left(\sum_{j=1}^n x_j t^{a_j}\right) dt\right)}.$$  We also prove that this asymptotic expansion $\Phi_\gamma(A,\beta;x)$ is a germ of Gevrey series in $\cO_{\widehat{X|Y}}(s)$  at any point in $Y\setminus Z$ for all $s\geq \frac{a_n}{a_{n-1}}$. Let us consider a vector space $E$ consisting in formal linear combinations of geometric cycles of the above type. We get a map $ \xymatrix{\gamma \ar[r]^{G_E\;}&\Phi _{\gamma}}$ from $E$ to the space of Gevrey series solutions of the system.
When $\beta \in \ZZ\setminus \NN A$, and restricting to a product of sectors in the variables $x_{n-1},x_n$ centered on half lines whose respective arguments are $(\pi-a_{n-1}\theta$ and $\pi-a_n\theta)$, we may as in the previous Section define an asymptotic expansion for the integral
\[
J_{\beta}(x_1,\dots,x_n)=\int_{e^{i\theta}\cdot[0,+\infty[}t^{-\beta-1}\left(e^{x_1t^{a_1}+\dots+x_nt^{a_n}}-P_{\beta}(x_1,\dots,x_n,t)\right)dt.
\]
Here $P_{\beta}(x_1,\dots,x_n,t)$ is again the Taylor polynomial of degree $\beta$ for $t\to\exp\left(x_1t^{a_1}+\dots+x_nt^{a_n}\right)$.
When $\beta \in \ZZ\setminus \NN A$ we still denote $\Phi_\gamma(A,\beta;x)$ such an asymptotic expansion. By extension we still call a vector space of cycles the formal direct sum of a one dimensional space $\CC\cdot e^{i\theta}\cdot[0,+\infty[$ and a subspace made of rapid decay cycles.

The goal of the remaining part of this Section is to prove the following generalisation of the surjectivity statement for $G_E$ that follows directly from Theorem \ref{a.e.are-basis-in-dim2}.
\begin{theorem}\label{a.e.are-basis-a1-an} There exists an $a_{n-1}$ dimensional vector  space $E$
of cycles such that the map $G_E$ from $E$ to the set of germs of Gevrey asymptotic expansions is an isomorphism onto the stalk of $\cH om_{\cD_X}\left(\cM_A(\beta), \cO_{\widehat{X|Y}}(s)\right)$ at any point in $Y\setminus Z$ for $s\geq \frac{a_n}{a_{n-1}}$.\end{theorem}

\begin{remark} 1) By \cite[Remark 4.12]{FC1} it is enough to prove the Theorem for points of the form $(0,\ldots,0,\varepsilon,0)\in X$ with $\varepsilon \not= 0$.

2) In the proof of this Theorem we will see that for $\beta \notin \ZZ \setminus \NN A$ the space $E$ is a space of rapid decay cycles. When $\beta \in \ZZ \setminus \NN A$ there exists a unique $e^{i\theta}\cdot [0,+\infty[$ which spans a one dimensional complement in $E$ to a space of rapid decay cycles.

3) In the statement of Theorem \ref{a.e.are-basis-a1-an} it must be understood that a Gevrey expansion can be obtained along any half line $e^{i\theta}\cdot [0,+\infty[$ in the variable $x_n$. The space $E$ depends on the arguments chosen for $x_{n-1},x_n$.

In the proof of Theorem \ref{a.e.are-basis-a1-an} we first focus on products of sectors in the variables $x_{n-1},x_n$ that are centered along the real negative axes, while the other variables $x_1,\dots,x_{n-2}$ are arbitrary in $\CC$. The general statement follows easily from this particular case by considering actions of roots of unity on the cycles and by controlling the width of the sectors of validity of the asymptotic expansions.
\end{remark}

\begin{corollary}\label{main-theorem-holds-for-dim-2} {\rm [of Theorem \ref{a.e.are-basis-in-dim2}]} Theorem \ref{a.e.are-basis-a1-an} holds for $n=2$ and a pair of sectors centered on the real negative axis for each of the variables $x_1,x_2$.\qed
\end{corollary}
%\begin{proof}
%In the proof of Theorem \ref{a.e.are-basis-in-dim2} one describes, by using Propositions \ref{DAsecteurs} and \ref{a.e.-gevrey-dim2-beta-generic},  how to build  cycles $\gamma_1, \ldots,\gamma_{a}$ and the corresponding asymptotic expansions $\Phi_1(A,\beta), \ldots, \Phi_a(A,\beta)$, for $A=(a,b)$, $1\leq a < b$, $\gcd(a,b)=1$. If $\beta \notin \ZZ \setminus \NN(a,b)$ all the $\gamma_\tau$ are rapid decay cycles. When $\beta \in \ZZ \setminus \NN(a,b)$ among them a unique cycle  is $[0,+\infty[$.
%\end{proof}

In order to simply notations we denote $\sol(\cM)=\cH om_{\cD}(\cM, \cO_{\widehat{X|Y}}(\frac{a_n}{a_{n-1}}))$ the sheaf of Gevrey solutions of order less than or equal to  $\frac{a_n}{a_{n-1}}$ of a holonomic $\cD_X$--module $\cM$ on $X=\CC^n$ for $Y=(x_n=0)\subset X$.

\begin{remark} \label{remark-corollary-4-10} By Corollary \ref{main-theorem-holds-for-dim-2},  for all $\beta \in \CC$ and for each $r=0,\ldots,k-1$ there exists a family $\gamma_{1,r}, \ldots,\gamma_{a,r}$ of cycles such that the family $$\left\{\Phi_{\gamma_{1,r}}\left((a,b),\frac{\beta -r}{k};x_1,x_2\right),\ldots, \Phi_{\gamma_{a,r}}\left((a,b),\frac{\beta -r}{k};x_{1},x_2\right)\right\}$$ is a basis of the stalk of the solution space $\sol(\cM_{(a,b)}(\frac{\beta-r}{k}))$ at any point in $Y_1 \setminus  Z_1$. Here $Y_1\subset \CC^2$ (resp. $Z_1\subset \CC^2$) is the line $x_2=0$ (resp. $x_1=0$).

In the statement of the following Proposition we use the morphism $\varpi_\beta:= \varpi_\beta^{s}$, for $s=\frac{a_n}{a_{n-1}}$,  defined in (\ref{varpi}). Recall that, according to Proposition \ref{varpi-for-k-geq-1} in Section 2, $\varpi_\beta$ is an isomorphism if $\beta\notin \NN \setminus \bigcup_{r=0}^{k-1}(r+k\NN(a,b))$. But if there exists $0\leq r_0 < k$ (necessarily unique) such that  $\frac{\beta -r_0}{k}\in \ZZ\setminus \NN(a,b)$ then one of the cycles $\gamma _{\tau_0, r_0}$ is $[0,+\infty[$. All the other cycles $\gamma _{\tau , r}$ are in the set of rapid decay ones and their asymptotic expansions span the image of $\varpi_\beta$ which contains no non zero polynomial.
\end{remark}
%Let us prove Theorem \ref{a.e.a1-an} for the matrix $(1,ka,kb)$. Assume first $\beta \not\in \ZZ$

\begin{proposition}\label{a.e.basis-1-ka-kb}  Theorem \ref{a.e.are-basis-a1-an}
holds for the matrix $(1, ka, kb)$ with $\gcd(a,b)=1$ and $k\geq 1$. There exists a set of cycles $\{\widetilde{\gamma}_{\tau,r}\, \vert\, 1\leq \tau \leq a, 0\leq r < k\}$, such that the asymptotic expansions $\Phi_{\widetilde{\gamma}_{\tau,r}}((1,ka,kb), \beta)$ span the space of solutions $\sol(\cM_{(1,ka,kb)}(\beta)).$ More precisely :

If $\varpi_\beta$ is an isomorphism, the image of the germ of $\Phi_{\widetilde{\gamma}_{\tau,r}}((1,ka,kb), \beta)$ by the morphism $\varpi_\beta$ equals $$\left(0,\cdots,0,\frac{1}{k}\Phi_{\gamma_{\tau,r}}\left((a,b),\frac{\beta-r}{k}\right),0,\cdots,0\right).$$

If $\varpi_\beta$ is not an isomorphism, the same is true for all the pairs $(\tau,r)$ such that  $\Phi_{\gamma_{\tau,r}}$ is in the image of $\varpi_\beta$. This exclude exactly one pair $(\tau_0,r_0)$ characterised by $\beta =r_0+kq_0$, and $\Phi_{\tau_0,r_0}$ not in the image of $\varpi_\beta$.
Furthermore $\Phi_{\widetilde{\gamma}_{\tau_0,r_0}}((1,ka,kb), \beta)$ is a polynomial.
\end{proposition}

\begin{proof} We consider the germs at a point $(0,\varepsilon,0)$ with $\varepsilon \not= 0$ and the stalk of the direct sum at $(\varepsilon, 0)$. We write $B'=(1,ka,kb)$ and $B=(a,b)$, and $(x_0,x_1,x_2)$ for coordinates in $\CC^3$. The proof here depends heavily on Proposition \ref{varpi-for-k-geq-1} in Section 2.
%%When $\beta\notin \ZZ$ or when $\beta\in \bigcup_{r=0}^{k-1}(r+k\NN(a,b))$ the morphism
%$\varpi_\beta^s$ studied in this Proposition is an isomorphism and all the cycles  that we obtain
%are rapid decay cycles. When $\beta \in \ZZ \setminus \bigcup_{r=0}^{k-1}(r+k\NN(a,b))$ one of
%these cycles is special. This last case contains the case $\beta \in \NN \setminus
%\bigcup_{r=0}^{k-1}(r+k\NN(a,b))$ where $\varpi_\beta^s$ is not an isomorphism.

Let us assume first that $\gamma_{\tau,r}$ is one of the rapid decay cycles among those considered in Remark \ref{remark-corollary-4-10}. They are relative to the matrix $B$ and $\frac{\beta-r}{k}$ and this excludes $[0,+\infty[$, when $\beta \in r_0+k(\ZZ\setminus\NN(a,b))$ for some $r_0\in \{0,\dots , k-1\}$. Denote it for short $\gamma:=\gamma_{\tau,r}$ and choose a $k^{\rm{th}}$ root $\widetilde{\gamma}$ in $\CC^*$ : $\widetilde{\gamma}(t)^k=\gamma(t)$.

 The cycle $\widetilde{\gamma}$ is of rapid decay with respect to $B'$ and we develop the integral $$I_{\widetilde{\gamma}}=I_{\widetilde{\gamma}}(x_0,x_1,x_2) = I_{\widetilde{\gamma}}(B',\beta;x_0,x_1,x_2)=\int_{\widetilde{\gamma}} t^{-\beta-1}\exp(x_0t+x_1t^{ka}+x_2t^{kb}) dt$$ as $$I_{\widetilde{\gamma}}(x_0,x_1,x_2)= \sum_{\ell=0}^{k-1} x_0^\ell J_{\widetilde{\gamma}}^\ell(x_0^k,x_1,x_2).$$ Notice that $\frac{\partial^\ell I_{\widetilde{\gamma}}}{\partial x_0^\ell}(0,x_1,x_2)=\ell !J_{\widetilde{\gamma}}^\ell(0,x_1,x_2)$ and the change of variables $s=t^k$ shows that :
\begin{align*}
I_{\gamma}\left(B, \frac{\beta-r}{k}; x_1,x_2\right):=&\int_{\gamma} s^{-\frac{\beta -r}{k}} \exp(x_1s^a + x_2 s^b) \frac{ds}{s} \\=& \int_{\widetilde{\gamma}} t^{-\beta +r} \exp(x_1t^{ka} + x_2 t^{kb}) k\frac{dt}{t}=k\frac{\partial^r I_{\widetilde{\gamma}}}{\partial x_0^r}(0,x_1,x_2)
\end{align*}

Let $\omega= e^{\frac{2 i \pi}{k}}$ and for $\nu=0,\ldots,k-1$, $\widetilde{\gamma}^{(\nu)}:= \omega^{-\nu}\widetilde{\gamma}$. We can write
$$I_{\widetilde{\gamma}}(x_0,x_1,x_2) = \omega^{-\nu\beta} \int_{\widetilde{\gamma}^{(\nu)}} s^{-\beta-1}\exp(x_0\omega^\nu s + x_1s^{ka}+x_2s^{kb}) ds = \omega^{-\nu\beta} I_{\widetilde{\gamma}^{(\nu)}}(x_0\omega^\nu, x_1,x_2)$$ and then $$I_{\widetilde{\gamma}^{(\nu)}} (x_0,x_1,x_2) = \omega^{\nu\beta} I_{\widetilde{\gamma}}(x_0 \omega^{-\nu}, x_1,x_2) = \omega^{\nu\beta} \sum_{\ell=0}^{k-1} \omega^{-\nu\ell} x_0^\ell J_{\widetilde{\gamma}}^\ell(x_0^k,x_1,x_2).$$ The matrix $(\omega^{-\nu\ell})_{0\leq \nu,\ell\leq k-1}$ being invertible we can write \begin{equation}\label{J-integrals} x_0^\ell J_{\widetilde{\gamma}}^\ell(x_0^k,x_1,x_2)= \sum_ {\nu=0}^{k-1} \lambda_{\ell,\nu} \omega^{-\nu\beta} I_{\widetilde{\gamma}^{(\nu)}}(x_0,x_1,x_2)\end{equation}
for some $\lambda_{\ell,\nu}\in \CC$.
Now the cycle $\widetilde{\gamma}_{\tau,r}:= \sum_{\nu=0}^{k-1} \lambda_{r,\nu}\omega^{-\nu\beta} \widetilde{\gamma}^{(\nu)}$ yields the expected result for $\gamma=\gamma_{\tau,r}$ :
 \[
 \begin{cases}\frac{\partial^\ell I_{\widetilde{\gamma}_{\tau,r}}}{\partial x_0^\ell}(0,x_1,x_2)=0
 \qquad\text{  if } \ell\not= r \\
\frac{\partial^r I_{\widetilde{\gamma}_{\tau,r}}}{\partial x_0^r}(0,x_1,x_2)=r! J^r_{\widetilde{\gamma}_{\tau,r}}(0,x_1,x_2) = \frac{1}{k}I_\gamma\left(B, \frac{\beta-r}{k}; x_1,x_2\right).
\end{cases}
\]

Consider the non rapid decay case $\gamma_{\tau_0,r_0}=[0,+\infty[$. This happens when $\beta \in \ZZ \setminus \bigcup_{r=0}^{k-1}(r+k\NN(a,b))$, $\beta = \beta' k+r_0$ for unique integers $\beta'$ and $0\leq r_0< k$. %Write by Euclidean division $\beta'=ha+p_0$, $0\leq p_0< a$.
%% We can apply the argument for $\beta \not\in \ZZ$ to any $(p,r) \not= (p_0,r_0)$ to prove that
%there exists a cycle $\widetilde{\gamma}_{p,r}$ such that the image of
%$\Phi_{\widetilde{\gamma}_{p,r}}(B',\beta)$ by the morphism $\varpi_\beta$ is $(0,\ldots,0,\frac{1}
%{k}\Phi_{\gamma_{p,r}}(B, \frac{\beta-r}{k}; x_1,x_2),0,\ldots,0)$.
We know by Theorem \ref{a.e.are-basis-in-dim2} that $\Phi_{\tau_0,r_0}$ is the asymptotic expansion attached to the integral :
$$J_{\frac{\beta-r_0}k,B}(x_1,x_2)=\int_{0}^{+\infty} s^{-\frac{\beta-r_0}k}(\exp(x_1s^a+x_2s^b) - P_{\frac{\beta-r_0}k}(x_1,x_2,s))\frac{ds}s$$
When $\varpi_\beta$ is an isomorphism, we consider the expression
$$J_{\beta,B'}(x_0,x_1,x_2)=\int_{0}^{+\infty} t^{-\beta-1}(\exp(x_0t+x_1t^{ka}+x_2t^{kb}) - P_\beta(x_0,x_1,x_2,t))dt$$
with notations analogous to those used in the proof of Theorem \ref{a.e.are-basis-in-dim2}. It can be proved as in Theorem \ref{a.e.are-basis-in-dim2} that %$J_{\frac{\beta-r}k,B}$
$J_{\beta,B'}$ is a holomorphic solution of the system $\cM_{B'}(\beta)$ and that it admits an asymptotic expansion which is of Gevrey order less than or equal to $\frac{b}{a}$ that we denote $\Phi_{[0,+\infty[}$. A straightforward calculation, with the change of variables $s=t^k$ shows that:
\[
\frac{\partial^{r_0} J_{\beta,B'}}{\partial x_0^{r_0}}(0,x_1,x_2)=J_{\frac{\beta-r_0}k,B}(x_1,x_2)
\]
so that the $r_0$-component of $\varpi_\beta(\Phi_{[0,+\infty[})$, is equal to $\Phi_{\tau_0,r_0}$.
Using the already determined asymptotic expansions for all others $\Phi_{\widetilde{\gamma}_{p,r}}$ we get the expected result with a uniquely determined linear combination $\widetilde{\gamma}_{\tau_0,r_0}=[0,+\infty[-\sum_{(\tau,r)\neq(\tau_0,r_0)}c_{\tau,r}\widetilde{\gamma}_{\tau,r}.$

If $\beta \in \NN $, the integral along the cycle $C_b$ (see notations in Subsection \ref{casab}) $$ \int_{C_b} t^{-\beta -1} \exp(x_0t+x_1t^{ka}+ x_2t^{kb})dt$$ is a polynomial. Therefore if $\beta \in r_0+k(\NN\setminus \NN(a,b))$ for some $r_0\in \{0\dots , k-1\}$, its image by $\varpi_\beta$ is zero and it is the missing asymptotic expansion in the solution space $\sol(\cM_{(1,ka,kb)}(\beta))$ by the final part of Remark \ref{remark-corollary-4-10}.
This finishes the proof.\end{proof}

\begin{proof} {\rm [of Theorem \ref{a.e.are-basis-a1-an}, general case]}. We may assume $n\geq 3$. We simply write $ka=a_{n-1}, kb=a_n$ for some integer $k\geq 1$ and $\gcd(a,b)=1$.

We first assume $a_1>1$ and denote as usual $A'=(1,a_1,\cdots,a_n)$. We also denote $B'=(1,a_{n-1},a_n)=(1,ka,kb)$ and $B=(a,b)$.  Let us consider the following diagram

$$\xymatrixcolsep{3pc}
\xymatrix{
\sol(\cM_{A'}(\beta))_p \ar[r]^{{\rho'}} \ar[d]^{{\rho}}
&\sol(\cM_{{B}'}(\beta))_p
%\ar[d]^{{\varpi_\beta}}
\\
\sol(\cM_{A}(\beta))_p.
%&\bigoplus_{r=0}^{k-1}\sol(\cM_{{B}}(\frac{\beta-r}{k}))_p
}$$

Let us explain the restriction morphisms in the above diagram:
$\rho'$ is the restriction defined as  $${\rho}'(f(x_0,x_1,\ldots,x_n))= f(x_0,0,\ldots,0,x_{n-1},x_n).$$ Similarly $\rho(f(x_0,x_1,\ldots,x_n))=f(0,x_1,\ldots,x_n)$.
%and $\varpi_\beta:=\varpi_\beta^s$ is the morphism defined in Subsection \ref{restriction-to-x0} for $s=\frac{b}{a}$.
We consider a point $p=(0,\dots,\varepsilon,0)\in \CC^{n+1}$ with $\varepsilon \not=0$ and we also denote $p$ the image of this point in the different considered spaces.

The morphism ${\rho'}$ is an isomorphism for any $ \beta\in \CC$, see Subsection \ref{case-a1-equal-1}.
The morphism ${\rho}$ is an isomorphism if $\beta \in \NN \setminus \NN A$, see Theorem \ref{basis-phi(0,x)-generic-case} and Remark \ref{basis-phi(0,x)-special-case}.%,  and $\varpi_\beta$ is an isomorphism if $\beta \not\in \bigcup_{r=0}^{k-1} (r+k(\NN \setminus \NN (a,b))$, see Proposition \ref{varpi-for-k-geq-1}.

Let us consider the set of asymptotic expansions $$\{\Phi_{\widetilde{\gamma}_{\tau,r}}(A',\beta)\, \vert\, 1\leq \tau \leq  a, 0\leq r < k\}$$ where $\widetilde{\gamma}_{\tau,r}$ is the cycle built in
the proof of Proposition \ref{a.e.basis-1-ka-kb}. The image by $\rho'$ of $\Phi_{\widetilde{\gamma}_{\tau,r}}(A',\beta)$ is just
$\Phi_{\widetilde{\gamma}_{\tau,r}}(B',\beta)$. This proves the theorem for $A'$ and then for $A$ if $\beta \not\in \NN \setminus \NN A$ (because in this case $\rho$ is an isomorphism and the image of $\Phi_{\widetilde{\gamma}_{\tau,r}}(A',\beta)$ is precisely $\Phi_{\widetilde{\gamma}_{\tau,r}}(A,\beta)$).

If $\beta \in \NN \setminus \NN A$ then we have again
$\rho(\Phi_{\widetilde{\gamma}_{\tau,r}}(A',\beta))=\Phi_{\widetilde{\gamma}_{\tau,r}}(A,\beta)$ for all $(\tau,r)\not=(\tau_0,r_0)$. The integral
\begin{equation} \label{jota-beta-a1-an} J_\beta := \int_0^{+\infty} t^{-\beta -1} \left(\exp\left(\sum_{i=1}^n x_it^{a_i}\right) -P_\beta (x_1,\ldots,x_n,t)\right)dt
\end{equation}
defines a holomorphic function in a domain $\CC^{n-2}\times S_{n-1}\times S_n$ where $S_i$ is a sector in $\CC$ which is a neighbourhood of the real negative axis for $i=n-1,n$. Here $P_\beta$ is the Taylor polynomial, for the exponential,  of degree $\leq\beta$ in $t$.

As usual we consider, for $k$ big enough,  \begin{equation} \label{c-k-jota-a1-an} c_{k}(x_1,\ldots,x_{n-1}) = \frac{1}{k !} \int_0^{+\infty} t^{-\beta -1+a_nk}\exp\left(\sum_{i=1}^{n-1} x_i t^{a_i}\right) dt \end{equation}
 the coefficient of $x_n^{k}$ in the expansion of $J_\beta$. By developing we get
$$c_{k}(x_1,\ldots,x_{n-1}) = \sum_{m_1,\ldots,m_{n-2} \geq 0} c_{m_1,\ldots,m_{n-2},k}(x_{n-1})x_1^{m_1} \cdots x_{n-2}^{m_{n-2}}$$ where, writing $m_n=k$ big enough,  $$c_{m_1,\ldots,m_{n-2},m_n}(x_{n-1}) = \frac{1}{m_1! \cdots m_{n-2}! m_n!} \int_{0}^{+\infty} t^{-\beta -1 + \sum_{i\neq n-1} a_i m_i} e^{x_{n-1}t^{a_{n-1}}} dt.$$ Up to a scalar multiple this last integral equals $$x_{n-1}^{\frac{\beta - \sum_{i\neq n-1} a_i m_i}{a_{n-1}}} \Gamma\left(\frac{-\beta + \sum_{i\neq n-1} a_i m_i}{a_{n-1}}\right).$$ The condition $\beta \in \NN \setminus \NN A$ implies that the argument of the Gamma factor is never a non-positive integer. Writing $\beta = q a_{n-1} + j_0$ with $0\leq j_0 < a_{n-1}$ and choosing $m_1,\ldots, m_{n-2}, m_{n-1},m_n \geq 0$ such that $j_0+a_{n-1}m_{n-1} = \sum_{i\neq n-1} a_i m_i$ we see that the corresponding exponent of $x_{n-1}$ in the expansion of $J_\beta$ is $\frac{\beta - \sum_{i\not= n-1} a_i m_i}{a_{n-1}} = q-m_{n-1}$ which is a negative integer if $m_{n-1}$ is large enough. Moreover,  the asymptotic expansion of $J_\beta$ is a Gevrey series solution of $\cM_A(\beta)$ of order less than or equal to $\frac{a_n}{a_{n-1}}$ which is linearly independent of the set $\{ \varphi_{A',\beta}^{(j)}(0,x)\, \vert \, j\not= j_0\}$.  This finishes the proof of the theorem for $A$ if $a_1>1$.

If finally $a_1=1$ then we apply previous discussion by using the restriction to the case $(1,a_{n-1},a_n)$.
\end{proof}

Up to now we have considered only asymptotic expansions in a neighbourhood of the real negative axes for $x_{n-1},x_n$. Looking at the reductions that we have carried out we see that it is sufficient to check the general statement about the arguments when $n=2$. We set $A=(a,b)$. If we consider $x_2=y_2e^{i\theta}$ of argument $\theta$ and changing all the cycles $\gamma$ into $e^{i\frac{\pi-\theta}{b}}\cdot \gamma$, we are reduced to the real negative case for $x_2$. Now we conclude with the following~:
\begin{lemma}
Theorem \ref{a.e.are-basis-a1-an} holds for $n=2$ in a product of sectors near the real negative axis for $x_2$ and an arbitrary argument for $x_1$.
\end{lemma}
\begin{proof}Changing the cycles $\gamma$ to $e^{2ip\pi/b}\cdot\gamma$ preserves the hypothesis for $x_2$
 and modify the argument of $x_1$ by a factor $e^{2ipa\pi/b}$. By varying $p$, and because $(a,b)=1$ we get all the $b$-roots of unity. It is therefore sufficient to check that the asymptotic expansion found in  Subsection \ref{casab} is valid in sectors around the real negative axis whose union for a given $\ell$ has a width strictly greater than $\frac{2\pi}{b}$. This follows from a careful inspection of  the proof of the enlargement statement in Proposition \ref{DAsecteurs}.
\end{proof}

\section{Gevrey solutions modulo convergent solutions} \label{gevrey-mod-convergent}
We can also give a description of the stalk of the solution space $\cH om_{\cD_X} (\cM_A(\beta), \cQ_Y(s))$ at any point of $Y\setminus Z$ where, for $s\geq 1$, $\cQ_Y(s)$ is the quotient of ${\cO_{\widehat{X\vert Y}}(s)}$ by ${\cO_{{X\vert Y}}}$.
By \cite[Th. 5.3]{FC1} this space is just $(0)$ if $1\leq s<\frac{a_n}{a_{n-1}}$ and has dimension $a_{n-1}$ if $s\geq \frac{a_n}{a_{n-1}}$. We will assume in this Section that $s\geq \frac{a_n}{a_{n-1}}$.
Let $\phi(t)$ a $C^{\infty}$ function with compact support locally constant with value 1 near the origin. We consider the following integral, see (\ref{jota-beta-a1-an}): \begin{equation} \label{jota-beta-phi} J_{\phi,\beta}(x)= J_{\phi,\beta}(x_1,\ldots,x_n) := \int_{0}^{+\infty} t^{-\beta -1} \left(e^{x_1 t^{a_1}+ \cdots + x_n t^{a_n}} - P_\beta(x,t\phi(t))\right)dt \end{equation} where $P_\beta(x,t)$ is the Taylor polynomial of the exponential of degree in $t$ less than or equal to $\beta$. We write $J_{\phi,\beta}(A; x)$ if we want to emphasize the dependence of this integral on the matrix $A$. This integral defines a holomorphic function in $\CC^{n-2}\times S_{n-1}\times S_n$ for some open sectors $S_i\subset \CC$, $i=n-1, n$,  each of them containing the real negative axis. In general this integral is not a solution of the hypergeometric system $\cM_A(\beta)$ but it is a solution modulo convergent power series.

Now we come back to the Gevrey solutions modulo convergent ones. We will treat the case $A=(a,b)$ first. According to Corollary \ref{main-theorem-holds-for-dim-2}, for $\beta \not \in \ZZ$ the family of asymptotic expansions $\{\Phi_\tau(A,\beta):= \Phi_{\gamma_\tau}(A,\beta), \tau=1,\ldots,a\}$ is a basis of Gevrey solutions  of $\cM_A(\beta)$ in $\cO_{\widehat{X\vert Y}}(s)$. For $\beta \not \in \ZZ$, the family $\{\psi_{A,\beta}^{(j)} \, \vert \, j=0,\ldots,a-1\}$ is also a basis of Gevrey solutions and their classes modulo convergent series form a basis of $\cH om_{\cD_X} (\cM_A(\beta), \cQ_Y(s))$ \cite[Th. 5.9 (i)]{FC2}. Then, the classes of $\{\Phi_\tau(A,\beta), \tau=1,\ldots,a\}$ modulo convergent series forms a basis of $\cH om_{\cD_X} (\cM_A(\beta), \cQ_Y(s))$.

%Previous situation is still valid when $\beta \in \ZZ \setminus \NN(a,b)$, taking into account that we have denoted $\Phi_{a}(A,\beta)$ the asymptotic expansion of an integral $J_\beta(x_1,x_2)$ over $\gamma=[0,\infty)$, see (\ref{jota-beta-a-b}).

Previous situation is still valid when $\beta \in \ZZ \setminus \NN(a,b)$, taking into account that we denote $\Phi_{a}(A,\beta)$ the asymptotic expansion of an integral $J_\beta(x_1,x_2)$ over $\gamma=[0,+\infty[$, see Theorem \ref{a.e.are-basis-in-dim2}. For $p=b$ the integer $\ell$ is just $a$ and the integral along  the cycle $C_b$ is zero. As explained in the proof of Theorem \ref{a.e.are-basis-in-dim2} the asymptotic expansion of $J_\beta(x_1,x_2)$ replace the missing integral along rapid decay cycles.

Assume now $\beta \in \NN(a,b)$ and write $\beta = j_0b+qa$ with $0\leq j_0 < a$ and $q\geq 0$. We know that the Gamma series $\psi^{(j_0)}_{A,\beta}$ is a polynomial and then its class modulo convergent series is zero. We denote by $\Phi_{a}(A,\beta)$ the asymptotic expansion  of $J_{\phi,\beta}(x_1,x_2)$, see (\ref{jota-beta-phi}). The coefficient $c_k(x_1)$ of this expansion is, for $k$ big enough,  exactly the same as in (\ref{c-k-jota-a-b}). In particular, the exponent of $x_1$ in $c_{am+j_0}(x_1)$ for $m$ big enough
is the negative integer $q-bm$. This proves that  the family $\{\Phi_{\tau}(A,\beta),\, \tau =1,\ldots, a\}$ is still linearly independent modulo convergent power series. Hence, it defines a basis of $\cH om_{\cD_X} (\cM_A(\beta), \cQ_Y(s))$.

We treat now the case $A=(a_1,\ldots,a_n)$ and $n\geq 3$. We write as usual $a_{n-1}=ka, a_n=kb$ for $k\geq 1$ and $\gcd(a,b)=1$. According to Theorem \ref{a.e.are-basis-a1-an}, for $\beta \not \in \NN$ the family of asymptotic expansions $\{\Phi_{\widetilde{\gamma}_{\tau,r}}(A,\beta),\, 1\leq \tau \leq a, \, 0\leq r < k \}$ is a basis of Gevrey solutions  of $\cM_A(\beta)$ in $\cO_{\widehat{X\vert Y}}(s)$. For $\beta \not \in \NN$, the family $\{\widetilde{\varphi}_{A,\beta}^{(j)}:=\varphi^{(j)}_{A',\beta}(0,x) \, \vert \, j=0,\ldots,a_{n-1}-1\}$ is also a basis of Gevrey solutions and their classes modulo convergent series form a basis of $\cH om_{\cD_X} (\cM_A(\beta), \cQ_Y(s))$ \cite[Th. 5.5, (i)]{FC1}. Then, the classes of $\{\Phi_{\widetilde{\gamma}_{\tau,r}}(A,\beta), \, 1\leq \tau \leq a, \, 0\leq r < k\}$ modulo convergent series form a basis of $\cH om_{\cD_X} (\cM_A(\beta), \cQ_Y(s))$.

Previous situation is still valid when $\beta \in \NN \setminus \NN A$, taking into account that we denoted $\Phi_{\widetilde{\gamma}_{\tau_0,r_0}}(A,\beta)$ the asymptotic expansion of an integral $J_\beta(x)$ over $[0,+\infty[$, see (\ref{jota-beta-a1-an}) and we considered this asymptotic expansion as the generator of a complement space of $\{\widetilde{\varphi}_{A,\beta}^{(j)}, \, j=0,\ldots,a_{n-1}-1\}$.

Assume now $\beta \in \NN A$.
%and write $\beta = p_0b+qa$ with $0\leq p_0 < a$ and $q\geq 0$. We know that the Gamma series $\phi^{(p_0)}_{A,\beta}$ is a polynomial and then its class modulo convergent series is zero.
We denote by $\Phi_{\widetilde{\gamma}_{\tau_0,r_0}}(A,\beta)$ the asymptotic expansion  of $J_{\phi,\beta}(x)$, see (\ref{jota-beta-phi}),  with respect to $x_n$. The coefficient $c_k(x_1,\ldots,x_{n-1})$ of this expansion is, for $k$ big enough,  exactly the same as in (\ref{c-k-jota-a1-an}). We can proceed as in the proof of the general case of Theorem \ref{a.e.are-basis-a1-an} to see that the classes modulo convergent power series of the asymptotic expansions $\Phi_{\widetilde{\gamma}_{\tau,r}}(A,\beta)$ form a basis of the solution space
$\cH om_{\cD_X} (\cM_A(\beta), \cQ_Y(s))$.

\section*{Acknowledgements} The first author would like to thank the D\'epartement de Math\'ematiques de l'Universit\'e d'Angers (France) and the Geanpyl project of the Pays de Loire for their support and hospitality during the preparation of the final version of this work.

The second author would like to thank the University of Sevilla (Spain) for its support during a stay in April 2012, for the preparation of this paper.

\end{document}